\documentclass[11pt,reqno]{amsart}

\usepackage[margin=1.1in]{geometry}
\usepackage{amsmath,amssymb,amsthm,mathtools}
\usepackage{enumitem}
\usepackage{booktabs}
\usepackage{xcolor}
\usepackage{array}
\usepackage[colorlinks=true,linkcolor=blue,citecolor=blue]{hyperref}
\usepackage{appendix}
\theoremstyle{plain}
\newtheorem{theorem}{Theorem}[section]
\newtheorem{proposition}[theorem]{Proposition}
\newtheorem{lemma}[theorem]{Lemma}
\newtheorem{corollary}[theorem]{Corollary}
\newtheorem{conjecture}[theorem]{Conjecture}

\theoremstyle{definition}
\newtheorem{definition}[theorem]{Definition}
\newtheorem{remark}[theorem]{Remark}

\DeclareMathOperator{\Ric}{Ric}
\DeclareMathOperator{\Area}{Area}
\DeclareMathOperator{\diam}{diam}
\DeclareMathOperator{\rank}{rank}
\DeclareMathOperator{\supp}{supp}
\DeclareMathOperator{\vol}{vol}

\newcommand{\R}{\mathbb{R}}
\newcommand{\Z}{\mathbb{Z}}
\newcommand{\Sph}{\mathbb{S}}
\newcommand{\Cc}{\mathcal{C}}
\newcommand{\eps}{\varepsilon}
\newcommand{\la}{\lambda}
\newcommand{\dd}{\,dA_0}
\newcommand{\round}{g_0}

\title{Minimizers of Laplace eigenvalues under a lower curvature bound}

\author{Aditya Tiwari}
\address{Department of Mathematical Sciences, Indian Institute of Science
Education and Research (IISER) Mohali, Sector 81, Knowledge City, S.A.S.~Nagar,
Manauli, Punjab 140306, India}
\email{adityatiwari@iisermohali.ac.in}


\begin{document}

\begin{abstract}
We prove a sharp comparison, with Obata-type rigidity, for all Neumann eigenvalues
of one-dimensional $\mathrm{CD}(1,2)$ spaces against the Legendre model, under a
convexity condition on the density. It follows that minimizers of the $k$-th Laplace
eigenvalue among closed surfaces of Gaussian curvature at least $1$ cannot collapse
in the measured Gromov--Hausdorff completion, for every $k\ge2$. We also give a
variational proof that smooth minimizers are round.
\end{abstract}

\maketitle
\section{Introduction}\label{sec:intro}

Let $(M,g)$ be a closed surface with Gaussian curvature $K\ge1$. The
Lichnerowicz--Obata theorem \cite{Lichnerowicz,Obata} gives $\lambda_1\ge2$ for the
first nonzero Laplace eigenvalue, with equality only for the round unit sphere. Is the
round sphere extremal for all eigenvalues, and not only the first?

To make this precise, let
\[
\Cc=\bigl\{\,(M,g)\ :\ M\text{ is a closed smooth surface},\ K_g\ge1\,\bigr\}
\]
and, for each $k\ge1$, let
\begin{equation}\label{eq:functional}
    f_k\colon (M,g) \longmapsto \lambda_k(M,g),
\end{equation}
the $k$-th Laplace eigenvalue, eigenvalues being listed with multiplicity. Those of
$(\Sph^2,\round)$ are $\ell(\ell+1)$ for $\ell\ge0$, with multiplicity $2\ell+1$, so
Lichnerowicz--Obata says that the round sphere minimizes $f_1$ on $\Cc$. We ask whether
it minimizes every $f_k$, that is,
\begin{equation}\label{eq:comparison}
\lambda_k(M,g)\ \ge\ \lambda_k(\Sph^2)\qquad\text{for all }(M,g)\in\Cc,\ k\ge1?
\end{equation}

In dimension $n$, the analogous question is whether
$\lambda_k(M^n)\ge\lambda_k(\Sph^n)$ whenever $\Ric\ge(n-1)g$. For $k=1$, the
Lichnerowicz--Obata theorem again yields the same conclusion. For higher
eigenvalues the picture is now complete. In dimensions four and above the
inequality fails; Donnelly \cite{Donnelly} gave a counterexample in dimension four, which is adaptable in every
dimension $n\ge4$. See also Aryan \cite{Aryan} for
counterexamples among metrics on $\Sph^n$. In dimension three the inequality
also fails, by an explicit conformal perturbation of the round metric on $\Sph^3$
due to Lin, Wang and Xu \cite[Thm.~16.1]{LinWangXu}. In dimension two, by
contrast, it holds; the same authors prove
$\la_i(\Sph^2,g)\ge\la_i(\Sph^2,\round)$ for every smooth $g$ with $K_g\ge1$,
with rigidity at any single index \cite[Thm.~1.8]{LinWangXu}, and establish a
counting comparison, with rigidity at a single round threshold, for Alexandrov
two-spheres of curvature $\ge1$ \cite[Thm.~1.10]{LinWangXu}. This problem has
also been studied for several families of manifolds in dimensions two and three
satisfying \eqref{eq:comparison}; see \cite{ATM2,ATM1,ATM3}.

Since $\Cc$ is not closed under measured Gromov--Hausdorff (mGH) convergence
(Remark~\ref{rem:nonclosed}), minimizers must be sought in its mGH completion.
Besides two-dimensional Alexandrov surfaces, including singular metrics such as
cone metrics and spindles, the completion contains collapsed Alexandrov spaces of
Hausdorff dimensions one and zero. Firstly, we prove a one-dimensional comparison theorem, which will be used to rule out collapsed spaces as minimizers beyond the first eigenvalue.

For a density $\varphi$ on
$[0,D]$ we write $X_\varphi=([0,D],|\cdot|,\varphi\,dt)$ and
$0=\mu_0(\varphi)<\mu_1(\varphi)<\cdots$ for the Neumann spectrum of
$\mathsf L_\varphi u=-\varphi^{-1}(\varphi u')'$ on $X_\varphi$.

\begin{theorem}\label{thm:intro-1d}
Let $D>0$ and let $\varphi\colon[0,D]\to[0,\infty)$ be continuous, positive on $(0,D)$,
and such that
\begin{equation}\label{eq:intro-density}
\varphi''+\varphi\ \le\ 0\qquad\text{in the sense of distributions on }(0,D).
\end{equation}
Then
\begin{equation}\label{eq:main-ineq}
\mu_j(\varphi)\;\ge\; j(j+1)\qquad\text{for every } j\ge 0 .
\end{equation}
If equality holds for some $j\ge1$, then $D=\pi$ and $\varphi$ is a positive constant
multiple of $\sin$, so that $X_\varphi$ is the Legendre model.
\end{theorem}
\begin{remark}
Writing $\varphi=e^{-f}$, condition \eqref{eq:intro-density} is equivalent to
$f''\ge1+(f')^2$, i.e.\ the curvature-dimension condition $\mathrm{CD}(1,2)$.
The Legendre model arises as a geometric limit of admissible surfaces under
collapse, so the constant $j(j+1)$ in \eqref{eq:main-ineq} is sharp; see
\S\ref{sec:spindle}.
\end{remark}

Theorem~\ref{thm:intro-1d}, together with a diameter bound for the
zero-dimensional case, rules out collapsed spaces as minimizers.

\begin{theorem} \label{thm:intro-non-collapse}
Let $k\ge2$. If $g$ in the mGH completion of $\Cc$ minimizes
$f_k$, then the underlying Alexandrov space has Hausdorff dimension two.
\end{theorem}

With collapse ruled out, we define $\overline{\Cc}$ to be the two-dimensional
part of the completion of $\Cc$, consisting of Alexandrov surfaces of curvature
$\ge1$, each equipped with its area measure; see \S\ref{sec:prelim}.

We first analyze the smooth part, namely $\Cc$, by a variational argument and identify
its minimizers.
\begin{theorem}\label{thm:intro-smooth}
Let $k\ge2$ and suppose $\inf_\Cc\la_k$ is attained by a smooth metric
$g_\star\in\Cc$. Then $g_\star$ is the round unit sphere, and
$\inf_\Cc\la_k=\la_k(\Sph^2)$.
\end{theorem}

We also prove the following rigidity result.
\begin{proposition}\label{prop:rigidity}
Let $g\in\overline\Cc$ satisfy $\la_{\ell^2}(g)=\ell(\ell+1)$ for every $\ell\ge1$.
Then $g$ is the round unit sphere.
\end{proposition}

Both of these last two results also follow from \cite{LinWangXu}; the proofs given here are variational
and independent of theirs.

Together with the existence of minimizers in the completion
(Corollary~\ref{cor:existence}), the only unresolved possibility is
that a minimizer is a genuinely singular Alexandrov surface. These results suggest
that, if such a minimizer exists, all curvature in excess of $1$ is singular. We
therefore conjecture the following.
\begin{conjecture}\label{conjecture}
Let $k\ge2$ and let $g\in\overline{\Cc}$ minimize $f_k$. Let
$\omega=\omega_{\mathrm{ac}}+\omega_{\mathrm{sing}}$ be the Lebesgue decomposition
of its curvature measure with respect to $dA_g$. Then $\omega_{\mathrm{ac}}=dA_g$;
equivalently, the Gaussian curvature density satisfies $K=1$ almost everywhere.
\end{conjecture}
\begin{remark}
    Note that the spindles $\Sph^2_\alpha$ all attain $\lambda_1=2$, so a minimizer of
$f_1$ need not be round; the conjecture concerns the curvature density, not
uniqueness. More generally, Theorem~\ref{thm:spindle-equality} shows that at each
cluster bottom $\ell^2$, the equality value $\ell(\ell+1)$ is attained by an
entire family of spindles, so the unit sphere is not the unique minimizer.
\end{remark}

\textbf{Idea of the Proof}. A minimizing sequence is precompact, but its limit may collapse. We first show that any collapsed limit is an interval carrying a density satisfying \eqref{eq:intro-density}. We then use Sturm--Liouville theory to rule out collapse, thereby ensuring that the limiting minimizer is genuinely two-dimensional.
The uniformization theorem
implies that for every $(M,g)\in\Cc$, there exists $u\in C^\infty(\Sph^2)$ such that
$g=e^{2u}\round$, turning the curvature constraint $K\ge1$ into
\[
1+\Delta_0u-e^{2u}\ge0.
\]
Assuming a smooth minimizer is not round, its curvature is strictly greater
than $1$ on an open set. One may therefore
perturb the conformal factor in both directions on this set. First-order
stationarity forces the eigenfunctions corresponding to the minimizing
eigenvalue to satisfy a linear relation there. Aronszajn's unique continuation
theorem \cite{Aronszajn} then propagates this relation to all of $\Sph^2$,
contradicting their linear independence.

\textbf{Organization.} Section~\ref{sec:prelim} fixes the class, the conformal
reformulation, and the Alexandrov completion $\overline\Cc$.
Section~\ref{sec:reduction} establishes the reduction to the indices $\ell^2$.
Section~\ref{sec:spindle} identifies the limit of a collapsing family as a weighted interval, poses the resulting singular Sturm--Liouville problem, and analyses the spherical spindles, which realize that limit and show the thresholds below to be sharp. In Section~\ref{sec:noncollapse}, we rule out geometric collapse and prove the existence of a minimizer within the completion. Section~\ref{sec:smooth} delivers the proof of smooth rigidity by a variational argument, and Section~\ref{sec:completion-rigidity} proves the simultaneous rigidity stated in Proposition~\ref{prop:rigidity}, by a Weyl-law argument independent of it. Finally, Section~\ref{sec:1d-comparison} proves Theorem~\ref{thm:intro-1d}, the sharp comparison with rigidity for the full Neumann spectrum of one-dimensional $\mathrm{CD}(1,2)$ densities. In Appendix~\ref{Appendix}, we prove Theorem~\ref{thm:slp_discrete}, and identify the self-adjoint realization generated by the Cheeger energy. 

\textbf{Acknowledgments.} The author thanks Dr.~Anandateertha Mangasuli for introducing him to these questions. While preparing the manuscript, the author became aware of the related work \cite{LinWangXu}.

\section{Preliminaries}\label{sec:prelim}
Throughout, $(M,g)$ denotes a closed connected smooth two-dimensional Riemannian
manifold, $K_g$ its Gaussian curvature, $dA_g$ its area measure, $d_g(\cdot,\cdot)$ the
distance function, and $\la_k(M,g)$ the eigenvalues of the nonnegative Laplace--Beltrami
operator $\Delta_g=-\operatorname{div}_g\nabla_g$, listed in nondecreasing order with
multiplicity,
\[
0=\la_0(M,g)<\la_1(M,g)\le\la_2(M,g)\le\cdots\nearrow\infty,
\]
so that $\la_1$ is the first nonzero eigenvalue. By the min--max characterization,
\begin{equation}\label{eq:minmax}
\la_k(M,g)=\min_{\substack{V\subset H^1(M)\\ \dim V=k+1}}\ \max_{0\ne\phi\in V}\
\frac{\displaystyle\int_M|\nabla_g\phi|^2\,dA_g}{\displaystyle\int_M\phi^2\,dA_g},
\end{equation}
the dimension $k+1$ accounting for the constant mode $\la_0=0$.

Two consequences of the curvature bound are used repeatedly. By Bonnet--Myers,
every $(M,g)\in\Cc$ has $\diam(M,g)\le\pi$. By
Gauss--Bonnet, $\int_M K_g\,dA_g=2\pi\chi(M)$, and $K_g\ge1>0$ gives $\chi(M)>0$, hence
$\chi(M)=2$ and $M\cong\Sph^2$; thus
\begin{equation}\label{eq:area-bound}
\Area(M,g)=\int_M dA_g\ \le\ \int_M K_g\,dA_g=4\pi,
\end{equation}
with equality if and only if $K_g\equiv1$.

\begin{remark}\label{rem:nonclosed}
The class $\Cc$ is not closed under mGH convergence, so a
minimizing sequence for $f_k$ need not converge within $\Cc$, and minimizers need not
exist there. Indeed, the round spheres $(S^2(r),g_r)$, $0<r\le1$, have $K_{g_r}=r^{-2}\ge1$ and so
lie in $\Cc$, while $\diam(S^2(r))=\pi r\to0$; normalizing their vanishing areas to unit
mass, they converge as $r\searrow0$ to the one-point space with its Dirac measure, which
is not a closed two-dimensional manifold.
\end{remark}

\subsection{The conformal description of the class}\label{sec:conformal}

The uniformization theorem replaces the curvature condition defining $\Cc$ by a single
inequality for a scalar function on the round sphere.

\begin{theorem}[\cite{Taylor}]\label{thm:uniformization}
Let $(M,g)$ be a closed two-dimensional Riemannian manifold diffeomorphic to $\Sph^2$.
Then there exist a diffeomorphism $M\cong\Sph^2$ and a function $u\in C^\infty(\Sph^2)$
with
\[
g=e^{2u}\round,
\]
and $u$ is unique up to the action of the conformal group of $(\Sph^2,\round)$.
\end{theorem}

Let $\Delta_0$ denote the nonnegative Laplacian of $(\Sph^2,\round)$ and $\dd$ its area
measure. Under $g=e^{2u}\round$ the curvature satisfies
\begin{equation}\label{eq:liouville}
K_g\,e^{2u}=K_{\round}+\Delta_0 u=1+\Delta_0 u,
\end{equation}
and the area measure is $dA_g=e^{2u}\dd$. Hence
\begin{equation}\label{eq:K-conformal}
K_g=e^{-2u}\bigl(1+\Delta_0 u\bigr),
\end{equation}
and the condition $K_g\ge1$ is equivalent to
\begin{equation}\label{eq:constraint}
1+\Delta_0 u-e^{2u}\ \ge\ 0\qquad\text{on }\Sph^2.
\end{equation}

Combining uniformization with \eqref{eq:constraint} describes $\Cc$ in terms of
functions on the round sphere.

\begin{proposition}\label{prop:conformal-class}
The map $u\mapsto e^{2u}\round$ is a bijection, modulo the conformal group of
$(\Sph^2,\round)$, between
\[
\mathcal{U}=\bigl\{\,u\in C^\infty(\Sph^2)\ :\ 1+\Delta_0 u-e^{2u}\ge0\ \text{on }\Sph^2\,\bigr\}
\]
and the class $\Cc$. Under this identification the area functional is
$\Area(e^{2u}\round)=\int_{\Sph^2}e^{2u}\dd$, and the eigenvalue functional
\eqref{eq:functional} is the $k$-th eigenvalue of the weighted problem
\begin{equation}\label{eq:weighted-eig}
\Delta_0\phi=\la\,e^{2u}\phi\qquad\text{on }\Sph^2 ,
\end{equation}
where $\Delta_0$ is the Laplacian on $(\Sph^2,\round)$.
\end{proposition}

\begin{proof}
Theorem~\ref{thm:uniformization} provides $u\in C^\infty(\Sph^2)$, unique up to the
conformal group, with $g=e^{2u}\round$, and \eqref{eq:K-conformal} turns $K_g\ge1$ into
\eqref{eq:constraint}; the construction is invertible. The area identity is immediate
from $dA_g=e^{2u}\dd$. In two dimensions $\Delta_g=e^{-2u}\Delta_0$, so
$\Delta_g\phi=\la\phi$ reads \eqref{eq:weighted-eig}; the Dirichlet energy
$\int|\nabla_g\phi|^2\,dA_g=\int|\nabla_0\phi|^2\dd$ is a conformal invariant while the
mass $\int\phi^2\,dA_g=\int\phi^2e^{2u}\dd$ carries the weight, so \eqref{eq:minmax}
realizes $\la_k$ as the $k$-th eigenvalue of \eqref{eq:weighted-eig}.
\end{proof}

\subsection{Alexandrov surfaces of curvature \texorpdfstring{$\ge1$}{}}\label{sec:alexandrov}

Let $(X,d)$ be a geodesic metric space. For a geodesic triangle
$\triangle abc\subset X$ of perimeter $<2\pi$, its \emph{comparison triangle}
$\triangle\bar a\bar b\bar c\subset\Sph^2$ is the triangle in the round unit sphere
with the same three side lengths, unique up to isometry; a point $x\in[b,c]$ is
\emph{matched} to the point $\bar x\in[\bar b,\bar c]$ at the same arclength from the
corresponding endpoint.

\begin{definition}[\cite{BGP}; see also \cite{AKP}]\label{def:alexandrov}
A geodesic metric space $(X,d)$ has \emph{curvature $\ge1$ in the sense of Alexandrov}
if every point has a neighborhood $U$ such that, for every geodesic triangle
$\triangle abc\subset U$ of perimeter $<2\pi$, every vertex $a$, and every point $x$ on
the opposite side $[b,c]$,
\begin{equation}\label{eq:triangle-comparison}
d(a,x)\ \ge\ d_{\Sph^2}(\bar a,\bar x),
\end{equation}
$\triangle\bar a\bar b\bar c$ being the comparison triangle and $\bar x$ the point
matched to $x$. A two-dimensional such space is an \emph{Alexandrov surface of curvature
$\ge1$}.
\end{definition}

A smooth surface satisfies Definition~\ref{def:alexandrov} exactly when $K\ge1$, so the
definition extends the condition defining $\Cc$ to the non-smooth setting.

Such a surface is one of \emph{bounded integral curvature} \cite{AZ,Reshetnyak}, with
\emph{curvature measure} $\omega$, finite on compact sets, of density $K\,dA$ on the
smooth part and with an atom $2\pi-\theta\ge0$ at each conical point of angle
$\theta\in(0,2\pi]$. Curvature $\ge1$ reads
\begin{equation}\label{eq:omega-bound}
\omega\ \ge\ dA\qquad\text{as measures},
\end{equation}
i.e.\ $K\ge1$ almost everywhere with nonnegative singular part, and Gauss--Bonnet implies
$\omega(X)=2\pi\chi(X)$. Hence $\omega(X)=2\pi\chi(X)\ge\Area(X)>0$ forces $\chi(X)=2$,
so $X\cong\Sph^2$; and by Reshetnyak's uniformization \cite{Reshetnyak} it is
conformally $e^{2u}\round$ with $\omega=(1+\Delta_0u)\,\dd$ distributionally, the
singular counterpart of Theorem~\ref{thm:uniformization}.

By the area and diameter bounds a sequence in $\Cc$ is precompact in the
Gromov--Hausdorff topology \cite{BGP}, and, after normalizing the measures, in the
mGH topology \cite{GMS}; its limits are compact Alexandrov
spaces of curvature $\ge1$, but of Hausdorff dimension only at most $2$ \cite{BGP}; the dimension drops
precisely when the area collapses, as when a spindle is pinched to a segment. Collapse
is ruled out for minimizing sequences in \S\ref{sec:noncollapse}; we name the
two-dimensional limits here.

\begin{definition}\label{def:Cbar}
Let $\overline\Cc$ be the class of two-dimensional Alexandrov surfaces of curvature
$\ge1$, each written $e^{2u}\round$ (the conformal representation being
Reshetnyak's \cite{Reshetnyak}, as recalled above), with curvature measure
$\omega=(1+\Delta_0 u)\,\dd\ \ge\ e^{2u}\dd$ of total mass $\omega(\Sph^2)=4\pi$. A
\emph{spherical cone metric} is a member of $\overline\Cc$ whose absolutely continuous
curvature density is $1$ almost everywhere, so that $\omega=e^{2u}\dd$ together with
finitely many atoms; the spindles \eqref{eq:spindle} are the cone metrics with two
antipodal atoms.
\end{definition}

\begin{proposition}\label{prop:rcd}
Every $g\in\overline\Cc$, equipped with its area measure
$\mathfrak m=e^{2u}\dd$, is a compact $\mathrm{RCD}(1,2)$ metric measure
space. Consequently, it is also an $\mathrm{RCD}(1,\infty)$ space.
Moreover, $\diam(g)\le\pi$ and $\Area(g)\le4\pi$.
\end{proposition}

\begin{proof}
A two-dimensional Alexandrov surface of curvature $\ge1$, with reference measure
$\mathcal H^2$, satisfies $\mathrm{CD}(1,2)$ \cite{Petrunin},
\cite[Appendix]{ZhangZhu}, and its Cheeger energy is a quadratic form. The canonical
Dirichlet form is constructed in \cite{KMS}, and on a compact Alexandrov space its
gradient flow coincides with that of the relative entropy \cite{GigliKuwadaOhta}, which
identifies it with the Cheeger energy. Hence $\mathrm{RCD}(1,2)$, and so
$\mathrm{RCD}(1,\infty)$. The
diameter bound is Bonnet--Myers for $\mathrm{CD}(1,2)$ \cite{Sturm}. For the area
bound the smooth identity \eqref{eq:area-bound} is unavailable on singular members;
instead \eqref{eq:omega-bound} and Gauss--Bonnet for surfaces of bounded integral
curvature give
\[
\Area(g)\ \le\ \omega(\Sph^2)\ =\ 2\pi\chi(\Sph^2)\ =\ 4\pi. \qedhere
\]
\end{proof}

The converse also holds. Every $\mathrm{RCD}(1,2)$ space $(X,d,\mathcal H^2)$ is a
two-dimensional Alexandrov surface of curvature $\ge1$ \cite{LytchakStadler}, so
$\overline\Cc$ is exactly the class of compact $\mathrm{RCD}(1,2)$ surfaces. Several of
the results quoted below are stated for the reduced condition
$\mathrm{RCD}^*(K,N)$; for finite $N$ the reduced and full conditions coincide
\cite{CavallettiMilman}, and we write $\mathrm{RCD}(1,2)$ throughout.

On $\mathrm{RCD}(1,\infty)$, the Laplacian is the nonnegative self-adjoint generator $\Delta_g$ of the
Cheeger energy $\int|\nabla f|^2\,d\mathfrak m$ on $L^2(\mathfrak m)$, well defined. On a smooth member $g=e^{2u}\round\in\Cc$ the
Cheeger energy is the classical Dirichlet energy, so $\Delta_g$ coincides with
$-\operatorname{div}_g\nabla_g$, and its eigenvalue problem is the weighted problem
\begin{equation}\label{eq:weighted-cbar}
\Delta_0\phi=\la\,e^{2u}\phi\qquad\text{on }\Sph^2.
\end{equation}
Thus $\Delta_g$ extends to $\overline\Cc$ the operator used on $\Cc$.

By Proposition~\ref{prop:rcd}, every $g\in\overline\Cc$ is $\mathrm{RCD}(1,\infty)$. Then the Cheeger-energy Laplacian corresponding to $g$ has discrete
spectrum; see \cite[Section~7]{GMS}, whose compactness hypothesis is supplied by the
positivity of $K=1$. On the
smooth members this recovers the classical discreteness of $-\operatorname{div}_g\nabla_g$
on $(\Sph^2,g)$. We record this, together with the accompanying spectral convergence; see \cite{AmbrosioHonda17}.

\begin{proposition}[\cite{GMS,AmbrosioHonda17,AmbrosioHonda}]\label{prop:discrete}
Every $g\in\overline\Cc$ has purely discrete Laplace spectrum
\begin{equation}\label{eq:spectrum-cbar}
0=\la_0(g)<\la_1(g)\le\la_2(g)\le\cdots\nearrow\infty,
\end{equation}
listed with finite multiplicities, with the min-max description \eqref{eq:minmax}, the
numerator read as the Cheeger energy.
If $g_n\to g$ in the mGH sense within
$\overline\Cc$, then $\la_k(g_n)\to\la_k(g)$ for every $k$, with multiplicity.
\end{proposition}

Therefore, the eigenvalue
functionals extend to the completion. For each $k\ge1$,
\begin{equation}\label{eq:functional-cbar}
f_k\colon\overline\Cc\longrightarrow(0,\infty),\qquad f_k(g)=\la_k(g),
\end{equation}
is well defined.

We next identify $\Cc$ with the two-dimensional part of mGH completion.

\begin{lemma}\label{lem:dense}
Every $g=e^{2u}\round\in\overline\Cc$ is a mGH limit of metrics
in $\Cc$.
\end{lemma}
\begin{proof}
Let $P_t=e^{-t\Delta_0}$ be the heat semigroup of $(\Sph^2,\round)$ and $u_t=P_t u$,
smooth for $t>0$. Since $P_t$ commutes with $\Delta_0$, preserves positivity, and fixes
constants,
\[
1+\Delta_0 u_t\;=\;P_t\bigl(1+\Delta_0 u\bigr)\;=\;P_t\,\omega\;\ge\;
P_t\bigl(e^{2u}\bigr)\;\ge\;e^{2P_t u}\;=\;e^{2u_t},
\]
the last inequality by Jensen's inequality, the heat kernel $p_t(x,\cdot)\dd$ being a
probability measure. Hence $e^{2u_t}\round\in\Cc$ for every $t>0$; the approximating
metrics are smooth with $K\ge1$ exactly, with no need to relax and rescale the
curvature bound. As $t\to0^+$ the potentials converge in $L^1$ with
$\Delta_0 u_t\to\Delta_0 u$ weakly; since every atom of $\omega$ has mass
$2\pi-\theta<2\pi$, the cone angles $\theta$ being positive, the distance functions
converge uniformly \cite[Theorem~7.3.1]{Reshetnyak}; see also \cite{ChenLi}.
Moreover $e^{2u_t}\le P_t(e^{2u})$ together with $P_t(e^{2u})\to e^{2u}$ in $L^1$ and
$u_t\to u$ almost everywhere gives, by generalized dominated convergence,
$e^{2u_t}\to e^{2u}$ in $L^1(\dd)$, i.e.\ convergence of the area measures weakly and
in total mass. Thus $e^{2u_t}\round\to g$ in the mGH sense.
Conversely, any two-dimensional limit of $\Cc$ is $\mathrm{RCD}(1,2)$. The condition
$\mathrm{CD}(1,2)$ is stable under mGH convergence \cite{Sturm} and infinitesimal
Hilbertianity is stable because $\mathrm{RCD}(1,\infty)$ is \cite{GMS}. Hence it lies in
$\overline\Cc$ by Proposition~\ref{prop:rcd}.
\end{proof}

\section{Reduction to the indices \texorpdfstring{$\ell^2$}{}}\label{sec:reduction}

The comparison \eqref{eq:comparison} at all indices is equivalent to its restriction to
the indices $\ell^2$.

\begin{lemma}\label{lem:block}
On $(\Sph^2,\round)$, indexed with multiplicity,
\begin{equation}\label{eq:block}
\la_k(\Sph^2)=\ell(\ell+1)\qquad\text{if and only if}\qquad \ell^2\le k\le(\ell+1)^2-1,
\end{equation}
for each integer $\ell\ge0$. In particular the value $\ell(\ell+1)$ first occurs at the
index $k=\ell^2$, and occurs with multiplicity $2\ell+1$.
\end{lemma}

\begin{proof}
The eigenvalues below $\ell(\ell+1)$ number $\sum_{j=0}^{\ell-1}(2j+1)=\ell^2$, so the
value $\ell(\ell+1)$ begins at index $\ell^2$; having multiplicity $2\ell+1$, it
occupies the indices $\ell^2,\dots,\ell^2+2\ell=(\ell+1)^2-1$.
\end{proof}

We call the indices $k=\ell^2$, at which the value $\ell(\ell+1)$ first appears, the
\emph{cluster bottoms}.

\begin{proposition}\label{prop:reduction}
For $(M,g)\in\Cc$, the following are equivalent.
\begin{enumerate}
\item[(i)] $\la_k(M,g)\ge\la_k(\Sph^2)$ for every $k\ge1$.
\item[(ii)] $\la_{\ell^2}(M,g)\ge\ell(\ell+1)$ for every $\ell\ge1$.
\end{enumerate}
Consequently the round sphere minimizes $f_k$ on $\Cc$ for every $k$ if and only if it
minimizes $f_{\ell^2}$ on $\Cc$ for every $\ell$.
\end{proposition}

\begin{proof}
That (i) implies (ii) is immediate. Conversely, given $k\ge1$ let $\ell$ be the unique
integer with $\ell^2\le k\le(\ell+1)^2-1$, so that $\la_k(\Sph^2)=\ell(\ell+1)$ by
Lemma~\ref{lem:block}. Since $k\ge\ell^2$ and the eigenvalues are nondecreasing in the
index,
\[
\la_k(M,g)\ \ge\ \la_{\ell^2}(M,g)\ \ge\ \ell(\ell+1)\ =\ \la_k(\Sph^2),
\]
the second inequality being (ii) at $\ell$.
\end{proof}

The non-collapse and existence results of Section~\ref{sec:noncollapse}, and the smooth
rigidity theorem, do not use this reduction; they are established at every index
$k\ge2$ directly.

\section{Geometric collapse}\label{sec:spindle}

The Legendre model arises among these as a sharp limit; the comparison theorem for
one-dimensional models is proved in Section~\ref{sec:1d-comparison}.

\begin{theorem}\label{thm:collapse-structure}
Let $(X_n)\subset\overline\Cc$ satisfy $\mathcal H^2(X_n)\to0$ and
$\diam(X_n)\to D>0$, and equip each with its normalized measure
$\mathfrak m_n=\mathcal H^2\llcorner X_n/\mathcal H^2(X_n)$. Then, after passing to a
subsequence, $(X_n,d_n,\mathfrak m_n)$ converges in the mGH
topology to
\[
X=\bigl([0,D],\,|\cdot|,\,c\,\varphi\,ds\bigr),
\]
where $c>0$ normalizes the mass and $\varphi$ is continuous on $[0,D]$, positive on
$(0,D)$, and satisfies $\varphi''+\varphi\le0$ distributionally, that is
\eqref{eq:intro-density}.
\end{theorem}
\begin{proof}
Each $X_n$ is a compact $\mathrm{RCD}(1,2)$ space by Proposition~\ref{prop:rcd},
a condition preserved by rescaling the reference measure. A subsequence converges
\cite{BGP,GMS}, and the limit $X$ is one-dimensional. It is a compact Alexandrov space
of curvature $\ge1$ and dimension at most $2$ \cite{BGP}, dimension $2$ would force
$\mathcal H^2(X)=\lim\mathcal H^2(X_n)=0$, and $D>0$ excludes a point. It is
$\mathrm{RCD}(1,2)$ with a probability reference measure, since $\mathrm{CD}(1,2)$ is
stable under mGH convergence \cite{Sturm} and infinitesimal
Hilbertianity is stable because $\mathrm{RCD}(1,\infty)$ is \cite{GMS}.

One-dimensional $\mathrm{RCD}(1,2)$ spaces are circles or intervals carrying
$c\,e^{-f}\mathcal H^1$ with $f$ $(1,2)$-convex \cite{KitabeppuLakzian}, and the
positive lower bound excludes the circle \cite[\S4]{KitabeppuLakzian}. So
$X=([0,D],c\,\varphi\,ds)$ with $\varphi=e^{-f}$ and $f''\ge1+(f')^2$. Convexity of $f$
makes $\varphi$ locally Lipschitz on $(0,D)$ and continuous up to the endpoints, and
$d\varphi'=-\varphi\,df'+(f')^2\varphi\,dt$ converts $f''\ge1+(f')^2$ into
$\varphi''+\varphi\le0$.
\end{proof}

\textbf{The Sturm--Liouville Problem.}  Theorem~\ref{thm:collapse-structure} produces a weighted interval, and with it the
singular Sturm--Liouville problem
\begin{equation}\label{eq:slp}
-(\varphi u')'=\mu\,\varphi u \quad\text{on }(0,D),
\qquad
\lim_{t\to0^+}\varphi u'=\lim_{t\to D^-}\varphi u'=0,
\end{equation}
where $\varphi\colon[0,D]\to[0,\infty)$ is continuous, positive on $(0,D)$, and satisfies $\varphi''+\varphi\le0$ distributionally.

In Appendix~\ref{Appendix}, we prove that this problem has a discrete, simple spectrum.
\begin{theorem} \label{thm:slp_discrete}
    Suppose $\varphi\colon[0,D]\to[0,\infty)$ is continuous, positive on $(0,D)$, and satisfies 
    \begin{equation}\label{eq:concavity}
\varphi\ \text{ is concave on }[0,D].
\end{equation}
     Then the eigenspectrum of the Sturm Liouville Problem~\eqref{eq:slp} is purely discrete and simple with \[
0=\mu_0(\varphi)<\mu_1(\varphi)<\mu_2(\varphi)<\cdots.
\]
\end{theorem}
\begin{remark}
The condition $\varphi''+\varphi\le0$ in the sense of distributions implies that
$\varphi$ is concave on $(0,D)$.
\end{remark}

Through an explicit example, we show that the spindles $\Sph^2_{\alpha}$ can also attain the eigenvalue values $\ell(\ell+1)$ at the cluster bottoms, and then explain why the Legendre model is dictated by the geometry rather than chosen arbitrarily.

\textbf{The spherical spindles}. For $0<\alpha\le1$, let
\begin{equation} \label{eq:spindle}
    g_\alpha=dt^2+\alpha^2\sin^2 t\,d\vartheta^2,\qquad
t\in[0,\pi],\ \vartheta\in[0,2\pi),
\end{equation}
a surface of revolution with warping factor $\alpha\sin t$, whose Gaussian
curvature is $1$ away from the poles $t\in\{0,\pi\}$, where the cone angle
is $2\pi\alpha$. With $\sqrt{\det g_\alpha}=\alpha\sin t$, the nonnegative Laplacian
acts by
\[
\Delta_{g_\alpha}u
=-\frac{1}{\alpha\sin t}\,\partial_t\!\bigl(\alpha\sin t\,\partial_t u\bigr)
-\frac{1}{\alpha^2\sin^2 t}\,\partial_\vartheta^2 u .
\]
The metric is invariant under rotation in $\vartheta$, so the eigenfunctions separate as
$u(t,\vartheta)=F(t)e^{im\vartheta}$ with $m\in\Z$, and exhaust the spectrum, each
$m$-mode having compact resolvent. The equation
$\Delta_{g_\alpha}u=\la u$ becomes the ordinary differential equation
\begin{equation}\label{eq:spindle-ode}
F''+\cot t\,F'+\Bigl(\la-\frac{m^2}{\alpha^2\sin^2 t}\Bigr)F=0,
\qquad t\in(0,\pi).
\end{equation}

Substitute $x=\cos t\in(-1,1)$. Writing $F(t)=P(x)$, equation \eqref{eq:spindle-ode}
transforms into the associated Legendre equation
\begin{equation}\label{eq:legendre}
(1-x^2)P''-2x\,P'+\Bigl(\nu(\nu+1)-\frac{\mu^2}{1-x^2}\Bigr)P=0,
\end{equation}
of degree $\nu$ and order
\begin{equation}\label{eq:order}
\mu=\frac{|m|}{\alpha},\qquad \la=\nu(\nu+1).
\end{equation}
Thus the eigenvalues are $\la=\nu(\nu+1)$ for the admissible values of $\nu$, which are
fixed by the realization at the two tips $x=\pm1$. For $m\ne0$ these are limit-point;
near $t=0$ the equation is Euler to leading order, with solutions behaving like
$t^{\pm\mu}$, $\mu=|m|/\alpha\ge1$, and $\int_0 t^{-2\mu}\,t\,dt$ diverges, so the
branch $t^{-\mu}$ is not in $L^2(\alpha\sin t\,dt)$ and no condition is needed. For
$m=0$ the tips are limit-circle and the Cheeger form imposes $\alpha\sin t\,F'\to0$,
excluding the logarithmic branch. Either way the admissible solutions are those bounded at $x=\pm1$.
Writing
$P(x)=(1-x^2)^{\mu/2}y(x)$ to factor out the regular boundary behavior,
\eqref{eq:legendre} becomes the Jacobi equation
\begin{equation}\label{eq:jacobi}
(1-x^2)y''-2(\mu+1)x\,y'+\bigl(\nu(\nu+1)-\mu(\mu+1)\bigr)y=0 .
\end{equation}
A solution of \eqref{eq:jacobi} bounded at both endpoints $x=\pm1$ exists if and only
if it is a polynomial, namely the Jacobi polynomial $y=P_n^{(\mu,\mu)}$ of degree
$n\in\Z_{\ge0}$ \cite[\S4.2]{Szego}, which occurs exactly when the constant term equals
the Jacobi eigenvalue,
\[
\nu(\nu+1)-\mu(\mu+1)=n(n+2\mu+1),
\]
that is, when
\begin{equation}\label{eq:nu-quant}
\nu=\mu+n=\frac{|m|}{\alpha}+n,\qquad n\in\Z_{\ge0}.
\end{equation}
Therefore, the eigenvalues of the spindle $\Sph^2_\alpha$ are
\begin{equation}\label{eq:spindle-spec}
\la_{m,n}=\nu(\nu+1),\qquad \nu=\frac{|m|}{\alpha}+n,\qquad m\in\Z,\ n\in\Z_{\ge0},
\end{equation}
with eigenfunction $F_{m,n}(t)e^{im\vartheta}$, where
$F_{m,n}(t)=(\sin t)^{\mu}P_n^{(\mu,\mu)}(\cos t)$ and $\mu=|m|/\alpha$. The
multiplicity of a value is the number of pairs $(m,n)$ producing it, with $m$ and $-m$
giving the same $\nu$.

For $\alpha=1$ this recovers the round sphere; $\nu=|m|+n\in\Z_{\ge0}$, the value
$\ell(\ell+1)$ arising from all pairs with $|m|+n=\ell$, of total multiplicity
$2\ell+1$. For $\alpha<1$ the value $2$ ($\nu=1$) is attained only by the pair
$(m,n)=(0,1)$ and is therefore simple, while the modes $m\ne0$ are displaced upward,
their values $\nu=|m|/\alpha+n$ increasing as $\alpha$ decreases.

We isolate the behavior at the cluster bottom $\ell^2$.

\begin{theorem}\label{thm:spindle-equality}
For every $\ell\ge1$ and $0<\alpha\le1$,
\begin{equation}\label{eq:spindle-equality}
\la_{\ell^2}(\Sph^2_\alpha)=\ell(\ell+1)\ \text{ for }\alpha\in\Bigl[\tfrac{\ell-1}{\ell},1\Bigr],
\qquad
\la_{\ell^2}(\Sph^2_\alpha)>\ell(\ell+1)\ \text{ for }0<\alpha<\tfrac{\ell-1}{\ell},
\end{equation}
where for $\ell=1$ the first interval is read as $(0,1]$.
\end{theorem}
\begin{proof}
For $0<\alpha\le1$ write $\nu(m,n)=|m|/\alpha+n$ and, counting with multiplicity over
$(m,n)\in\Z\times\Z_{\ge0}$ with $(m,n)\neq(0,0)$,
\[
N(\alpha)\;=\;\#\bigl\{(m,n):\ 0<\nu(m,n)<\ell\bigr\},
\qquad
S(\alpha)\;=\;\#\bigl\{(m,n):\ 0<\nu(m,n)\le\ell\bigr\}.
\]
By \eqref{eq:spindle-spec}, $N(\alpha)$ is the number of nonzero eigenvalues strictly
below $\ell(\ell+1)$ and $S(\alpha)$ the number of nonzero eigenvalues at most
$\ell(\ell+1)$, both with multiplicity. Since the nonzero eigenvalues are listed in
nondecreasing order,
\begin{equation}\label{eq:spindle-criterion}
\begin{aligned}
\la_{\ell^2}(\Sph^2_\alpha)=\ell(\ell+1)
&\iff N(\alpha)\le\ell^2-1 \ \text{ and } \ S(\alpha)\ge\ell^2,\\
\la_{\ell^2}(\Sph^2_\alpha)>\ell(\ell+1)
&\iff S(\alpha)\le\ell^2-1.
\end{aligned}
\end{equation}

\emph{The range $\tfrac{\ell-1}{\ell}<\alpha\le1$.}
A mode $(m,n)$ with $|m|+n\le\ell-1$ and $|m|\ge1$ satisfies
$\nu=\bigl(|m|+n\bigr)+|m|\bigl(\tfrac1\alpha-1\bigr)$, which equals $\ell$ exactly at
the threshold
\[
\alpha_{m,n}\;=\;\frac{|m|}{\ell-n}.
\]
Over $|m|+n\le\ell-1$ with $|m|\ge1$, this is maximal only at $(\pm(\ell-1),0)$, where
it equals $\tfrac{\ell-1}{\ell}$; for fixed $n$ it increases in $|m|$, and its extreme
value $(\ell-1-n)/(\ell-n)$ decreases in $n$. As $\nu(m,n)$ is strictly decreasing in
$\alpha$ for $m\neq0$, every mode with $0<|m|+n\le\ell-1$ therefore has $\nu<\ell$
throughout $\alpha\in(\tfrac{\ell-1}{\ell},1]$, while $|m|+n\ge\ell$ gives
$\nu\ge|m|+n\ge\ell$ for every $\alpha\le1$. Hence
\[
N(\alpha)\;=\;\#\{(m,n)\neq(0,0):|m|+n\le\ell-1\}\;=\;\ell^2-1 ,
\]
and the radial mode $(0,\ell)$ gives $S(\alpha)\ge N(\alpha)+1=\ell^2$, so
$\la_{\ell^2}(\Sph^2_\alpha)=\ell(\ell+1)$ by \eqref{eq:spindle-criterion}.

\emph{The threshold $\alpha=\tfrac{\ell-1}{\ell}$, $\ell\ge2$.} Here the pair
$(\pm(\ell-1),0)$ sits exactly at $\nu=\ell$ and all other modes with
$0<|m|+n\le\ell-1$ still satisfy $\nu<\ell$, so $N(\alpha)=\ell^2-3$; together with
$(\pm(\ell-1),0)$ and the radial mode $(0,\ell)$,
\[
S(\alpha)\;\ge\;(\ell^2-3)+3\;=\;\ell^2 ,
\]
and again $\la_{\ell^2}(\Sph^2_\alpha)=\ell(\ell+1)$ by
\eqref{eq:spindle-criterion}.

\emph{The range $\alpha<\tfrac{\ell-1}{\ell}$.} We count $S(\alpha)$ directly,
including the nonradial modes that may sit exactly at $\nu=\ell$. For
$\mu=|m|\ge0$, the values of $n$ with
$\nu(\mu,n)\le\ell$ are $0\le n\le \ell-\lceil\mu/\alpha\rceil$, which is nonempty only
for $\mu\le\alpha\ell$; each $\mu\ge1$ contributes twice ($m=\pm\mu$). Hence, excluding
the constant mode,
\[
S(\alpha)\;=\;\ell\;+\;2\sum_{\mu=1}^{\lfloor\alpha\ell\rfloor}
\Bigl(\ell+1-\bigl\lceil\tfrac{\mu}{\alpha}\bigr\rceil\Bigr).
\]
Since $\alpha<1$ gives $\lceil\mu/\alpha\rceil\ge\mu+1$, each summand is at most
$\ell-\mu$; and $\alpha\ell<\ell-1$ bounds the range of summation by $\mu\le\ell-2$.
Therefore
\[
S(\alpha)\;\le\;\ell+2\sum_{\mu=1}^{\ell-2}(\ell-\mu)
\;=\;\ell+2\Bigl(\frac{\ell(\ell-1)}{2}-1\Bigr)
\;=\;\ell^2-2\;\le\;\ell^2-1 ,
\]
so $\la_{\ell^2}(\Sph^2_\alpha)>\ell(\ell+1)$ by \eqref{eq:spindle-criterion}. This
proves \eqref{eq:spindle-equality}.
\end{proof}

Collapsing the spindles produces the Legendre model itself.
We exhibit it through a computation below.
\begin{corollary}\label{cor:spindle-collapse}
For every $\alpha\in(0,1]$, the radial part of the spectrum
\eqref{eq:spindle-spec}, namely the eigenvalues with $m=0$, is
$\{j(j+1):j\ge0\},$
independently of $\alpha$, with eigenfunctions $P_j(\cos t)$. Moreover,
$\Area(\Sph^2_\alpha)=4\pi\alpha,$
and, with normalized area measures,
\[
\Bigl(\Sph^2_\alpha,d_{g_\alpha},
\frac{\vol_{g_\alpha}}{\Area(\Sph^2_\alpha)}\Bigr)
\longrightarrow
\Bigl([0,\pi],|\cdot|,\tfrac12\sin t\,dt\Bigr)
\]
in the mGH topology as $\alpha\to0$. The limit is the Legendre model up to
the constant factor $\tfrac12$ in the measure, which does not affect its
operator or spectrum.
\end{corollary}

\begin{proof}
For $m=0$, \eqref{eq:nu-quant} gives $\nu=n$, hence
\[
\lambda_{0,n}=n(n+1),
\qquad
F_{0,n}=P_n(\cos t),
\]
independently of $\alpha$. Also
\[
\Area(\Sph^2_\alpha)
=
\int_0^{2\pi}\int_0^\pi\alpha\sin t\,dt\,d\vartheta
=
4\pi\alpha.
\]

Let $p_\alpha:\Sph^2_\alpha\to[0,\pi]$ be the projection
$p_\alpha(t,\vartheta)=t$. Since
\[
g_\alpha=dt^2+\alpha^2\sin^2t\,d\vartheta^2,
\]
we have
\[
|t_x-t_y|
\le d_{g_\alpha}(x,y)
\le |t_x-t_y|+\pi\alpha.
\]
Thus $p_\alpha$ is a surjective $\pi\alpha$-Gromov--Hausdorff
approximation. Moreover,
\[
(p_\alpha)_\#
\Bigl(\frac{\vol_{g_\alpha}}{4\pi\alpha}\Bigr)
=
\frac{2\pi\alpha\sin t}{4\pi\alpha}\,dt
=
\tfrac12\sin t\,dt.
\]
Hence the normalized spaces converge in the mGH topology to
$([0,\pi],|\cdot|,\tfrac12\sin t\,dt)$.

Finally, the limiting operator is
\[
-\frac1{\sin t}\frac{d}{dt}
\left(\sin t\,\frac{d}{dt}\right),
\]
since the constant factor $\tfrac12$ in the density cancels. Its natural
boundary conditions are those at the two endpoints, and its spectrum is
the radial spectrum above, namely $\{j(j+1):j\ge0\}$.
\end{proof}

\section{Non-collapse of minimizing sequences}\label{sec:noncollapse}

\begin{theorem}\label{thm:noncollapse}
Let $k\ge2$ and let $(g_n)\subset\Cc$ satisfy
$\Area(\Sph^2,g_n)\to0$. Then
\begin{equation}\label{eq:collapse-cost}
\liminf_{n\to\infty}\lambda_k(\Sph^2,g_n)
\ge k(k+1)>\lambda_k(\Sph^2).
\end{equation}
Consequently, no minimizing sequence for $f_k$ collapses.
\end{theorem}

\begin{proof}
Equip $(\Sph^2,g_n)$ with the normalized measure
$\mathfrak m_n=\vol_{g_n}/\Area(g_n)$. This does not change the curvature-dimension
condition or the Laplace spectrum, so $(\Sph^2,d_{g_n},\mathfrak m_n)$ is a compact
$\mathrm{RCD}(1,2)$ space. Passing to a subsequence, assume
\[
\diam(\Sph^2,g_n)\to D\in[0,\pi].
\]

If $D=0$, the Zhong--Yang estimate \cite{ZhongYang} gives
\[
\lambda_k(g_n)\ge\lambda_1(g_n)
\ge\frac{\pi^2}{\diam(\Sph^2,g_n)^2}\longrightarrow\infty.
\]
Thus suppose $D>0$. By Theorem~\ref{thm:collapse-structure}, after passing to a further
subsequence,
\[
(\Sph^2,d_{g_n},\mathfrak m_n)
\longrightarrow
X=([0,D],c\,\varphi\,ds)
\]
in the mGH sense, where $\varphi$ satisfies \eqref{eq:intro-density}. mGH convergence
of compact $\mathrm{RCD}(1,2)$ spaces implies Mosco convergence of the Cheeger
energies \cite{GMS,AmbrosioHonda17}, hence convergence of the spectra with
multiplicity \cite{AmbrosioHonda17,AmbrosioHonda} within the variational framework
of \cite{KuwaeShioya}, so
\[
\lambda_k(g_n)\longrightarrow\lambda_k(X).
\]
The limit being one-dimensional, this is the general $\mathrm{RCD}$ convergence and
not Proposition~\ref{prop:discrete}, which concerns limits within $\overline\Cc$.
By Lemma~\ref{lem:cheeger-identification} and
Theorem~\ref{thm:intro-1d},
\[
\lambda_k(X)\ge\mu_k(\varphi)\ge k(k+1).
\]
Since every subsequence admits such a further subsequence, \eqref{eq:collapse-cost}
follows.

Finally, if $\ell^2\le k\le(\ell+1)^2-1$, then, by Lemma~\ref{lem:block},
\[
\lambda_k(\Sph^2)=\ell(\ell+1),
\]
and $k\ge2$ implies $\ell<k$. Hence
\[
k(k+1)>\ell(\ell+1)=\lambda_k(\Sph^2).
\]
\end{proof}

Non-collapse yields existence of a minimizer in the completion, since the eigenvalue
functional is continuous there by Proposition~\ref{prop:discrete}.

\begin{corollary}\label{cor:existence}
For every $k\ge2$, the infimum
\[
I_k:=\inf_\Cc\lambda_k
\]
is attained by some $g\in\overline\Cc$, and
\[
I_k=\inf_{\overline\Cc}\lambda_k\le\lambda_k(\Sph^2).
\]
\end{corollary}

\begin{proof}
Let $(g_n)\subset\Cc$ be a minimizing sequence. By
\eqref{eq:area-bound}, Bonnet--Myers, and mGH compactness
\cite{BGP,GMS}, after passing to a subsequence,
\[
g_n\longrightarrow X.
\]
Since
\[
\lambda_k(g_n)\longrightarrow I_k
\le\lambda_k(\Sph^2)<k(k+1),
\]
Theorem~\ref{thm:noncollapse} rules out collapse. Hence $X$ is two-dimensional,
so $X\in\overline\Cc$. Proposition~\ref{prop:discrete} gives
\[
\lambda_k(X)=\lim_{n\to\infty}\lambda_k(g_n)=I_k.
\]
Thus $X$ attains the infimum. Finally, density of $\Cc$ in $\overline\Cc$
(Lemma~\ref{lem:dense}) and continuity of $\lambda_k$ give
\[
\inf_{\overline\Cc}\lambda_k=I_k.
\]
\end{proof}

\begin{remark}\label{rem:fk-extension}
For Theorem~\ref{thm:intro-non-collapse}, the functional $f_k$ is understood on the
full mGH completion of $\Cc$. After normalization, every element of the completion is
a compact $\mathrm{RCD}(1,2)$ space, so its Cheeger Laplacian has discrete spectrum
\cite[\S7]{GMS}, and $f_k$ is defined by \eqref{eq:minmax}. On the one-point space,
$L^2(\mathfrak m)$ is one-dimensional, so for $k\ge1$ no $(k+1)$-dimensional subspace
exists and $f_k=+\infty$.
\end{remark}

\begin{proof}[Proof of Theorem~\ref{thm:intro-non-collapse}]
Let $g$ minimize $f_k$ in the mGH completion. Since the round sphere is admissible,
\[
\lambda_k(g)\le\lambda_k(\Sph^2)<k(k+1).
\]
If the underlying Alexandrov space had dimension at most one, then either it is a
point, in which case $f_k(g)=+\infty$ by Remark~\ref{rem:fk-extension}, or it is a
nontrivial one-dimensional limit of a collapsing sequence. In the latter case,
Theorem~\ref{thm:collapse-structure}, Lemma~\ref{lem:cheeger-identification}, and
Theorem~\ref{thm:intro-1d} give
\[
\lambda_k(g)\ge\mu_k(\varphi)\ge k(k+1),
\]
a contradiction. Hence $g$ is two-dimensional.
\end{proof}

\section{Smooth minimizers are round}\label{sec:smooth}

Throughout this section, we assume that $g_\star=e^{2u}\round$ is a smooth
metric in $\Cc$ attaining $\inf_\Cc\lambda_k$ for a fixed $k\ge2$, and denote
its Gaussian curvature by $K_\star$.

\begin{lemma}\label{lem:slack}
Suppose $K_\star\not\equiv1$. Then the set $U=\{p\in\Sph^2:K_\star(p)>1\}$ is nonempty
and open, and there exist a nonempty open ball $V$ with compact closure
$\overline V\subset U$ and a constant $\delta>0$ such that
\[
K_\star\ \ge\ 1+\delta\qquad\text{on }\overline V .
\]
\end{lemma}

\begin{proof}
$K_\star$ is continuous, so $U=K_\star^{-1}\bigl((1,\infty)\bigr)$ is open; and $U$ is
nonempty, since $U=\varnothing$ would give $K_\star\le1$, hence $K_\star\equiv1$ with
the constraint. Choose $p\in U$ and set $\delta=\tfrac12(K_\star(p)-1)>0$. By
continuity there is $r>0$ with $K_\star\ge1+\delta$ on the closed coordinate ball
$\overline B(p,r)$, which is then compact and contained in $U$; take $V=B(p,r)$.
\end{proof}

We recall the constraint functional from \S\ref{sec:conformal}. Writing
\begin{equation}\label{eq:constraint-fnl}
\mathcal F(v)=1+\Delta_0 v-e^{2v},
\end{equation}
a conformal factor $v$ lies in the admissible set $\mathcal U$, equivalently
$e^{2v}\round\in\Cc$, precisely when $\mathcal F(v)\ge0$ on $\Sph^2$
(Proposition~\ref{prop:conformal-class}); and by \eqref{eq:K-conformal},
\begin{equation}\label{eq:F-curv}
\mathcal F(v)=e^{2v}\bigl(K_v-1\bigr),
\end{equation}
where $K_v$ is the curvature of $e^{2v}\round$. In particular
$\mathcal F(u)=e^{2u}(K_\star-1)\ge0$, with $\mathcal F(u)\ge\delta\,e^{2u}$ on
$\overline V$ by Lemma~\ref{lem:slack}.
\begin{lemma}\label{lem:twosided}
Suppose $K_\star\not\equiv1$, and let $V,\delta$ be as in Lemma~\ref{lem:slack}. For
every $v\in C_c^\infty(V)$ there exists $\eps_0>0$ such that
\[
u+\eps v\in\mathcal U
\qquad\text{for all }\eps\in[-\eps_0,\eps_0].
\]
\end{lemma}

\begin{proof}
Let $S=\supp v\subset V$. Outside $S$ the perturbation vanishes, so
$\mathcal F(u+\eps v)=\mathcal F(u)\ge0$. On $S$,
\[
\mathcal F(u+\eps v)
=
\mathcal F(u)+\eps\Delta_0v
-e^{2u}(e^{2\eps v}-1).
\]
By Lemma~\ref{lem:slack},
\[
\mathcal F(u)=e^{2u}(K_\star-1)
\ge\delta e^{2u}
\ge\delta m_0,
\qquad
m_0:=\min_{\overline V}e^{2u}>0.
\]
Set
\[
M_1:=\|\Delta_0v\|_\infty,
\qquad
M_2:=2\|e^{2u}\|_{L^\infty(V)}
\|v\|_\infty e^{2\|v\|_\infty}.
\]
Using $|e^x-1|\le |x|e^{|x|}$, for $|\eps|\le1$ we obtain
\[
\mathcal F(u+\eps v)
\ge
\delta m_0-(M_1+M_2)|\eps|.
\]
Thus the conclusion follows by taking
\[
\eps_0
=
\min\left\{
1,\frac{\delta m_0}{M_1+M_2}
\right\}.
\]
\end{proof}

The perturbation family yields a linear relation among the eigenfunctions on the
slack region. Fix an $L^2(e^{2u}\dd)$-orthonormal basis
$\phi^{(1)},\dots,\phi^{(m)}$ of the $\lambda$-eigenspace, where
$\lambda=\lambda_k(g_\star)$ has multiplicity $m$. Write
$r\in\{1,\dots,m\}$ for the position of $k$ in its eigenvalue block, and for
$v\in C_c^\infty(V)$ define
\begin{equation}\label{eq:Mmatrix}
M[v]_{ij}
=
\int_{\Sph^2}v\,\phi^{(i)}\phi^{(j)}e^{2u}\dd,
\qquad 1\le i,j\le m.
\end{equation}

\begin{lemma}\label{lem:relation}
Suppose $K_\star\not\equiv1$, with $V,\delta$ as in Lemma~\ref{lem:slack}. Then there
exist an open ball $V'\subset V$ and a vector $0\ne c=(c_1,\dots,c_m)\in\R^m$ such that
the eigenfunction combination
\[
\psi=\sum_{i=1}^m c_i\,\phi^{(i)}
\]
vanishes identically on $V'$.
\end{lemma}

\begin{proof}
Fix $v\in C_c^\infty(V)$, $v\ge0$. By Lemma~\ref{lem:twosided},
$u_\varepsilon=u+\varepsilon v$ is admissible for $|\varepsilon|\le\varepsilon_0$.
The substitution $w=e^{u_\varepsilon}\phi$ transforms the weighted eigenvalue
problem into the self-adjoint holomorphic family
\[
T_\varepsilon=e^{-u_\varepsilon}\Delta_0e^{-u_\varepsilon}
\]
on $L^2(\dd)$ with fixed form domain. Let $\lambda=\lambda_k(g_\star)>0$ have
multiplicity $m$, and let $w^{(i)}=e^u\phi^{(i)}$ be an orthonormal basis of its
eigenspace. By Rellich--Kato perturbation theory
\cite[Chapters~II, VII]{Kato}, the first derivatives of the $m$ eigenvalue
branches issuing from $\lambda$ are the eigenvalues of the compression of
$\dot T_0$ to this eigenspace.

Differentiating gives
\[
\dot T_0w=-vT_0w-e^{-u}\Delta_0(ve^{-u}w),
\]
and hence
\begin{equation}\label{eq:firstvar}
\bigl\langle\dot T_0w^{(i)},w^{(j)}\bigr\rangle_{L^2(\dd)}
=
-2\lambda\int_{\Sph^2}v\,\phi^{(i)}\phi^{(j)}e^{2u}\dd
=:A[v]_{ij}.
\end{equation}
Thus
\[
A[v]=-2\lambda M[v],
\qquad
M[v]_{ij}:=
\int_{\Sph^2}v\,\phi^{(i)}\phi^{(j)}e^{2u}\dd,
\]
and $M[v]$ is positive semidefinite.

Write
\[
\lambda_{k-r}<\lambda=\lambda_{k-r+1}=\cdots=\lambda_{k-r+m}
<\lambda_{k-r+m+1},
\]
with the obvious modifications at the ends of the spectrum. 
By continuity, for sufficiently small $\varepsilon$, the $m$ branches issuing from
$\lambda$ occupy the indices
$k-r+1,\ldots,k-r+m$. Writing
$\Lambda_j(\varepsilon)
=
\lambda+\varepsilon a_j+O(\varepsilon^2),$
where $a_1,\ldots,a_m$ are the eigenvalues of $A[v]$, and letting
$a_{(1)}\le\cdots\le a_{(m)}$
denote their increasing rearrangement, we have
\[
\lambda_k(u_\varepsilon)
=
\lambda+\varepsilon a_{(r)}+o(\varepsilon)
\qquad (\varepsilon\downarrow0).
\]
Hence the right derivative of $\lambda_k(u_\varepsilon)$ at $0$ is
$a_{(r)}$.

Since
$g_\star$ minimizes $\lambda_k$,
\[
\lambda_k(u_\varepsilon)\ge\lambda_k(u)
\qquad(\varepsilon\ge0),
\]
and this derivative is therefore nonnegative. Since $A[v]$ is negative
semidefinite, its $r$-th smallest eigenvalue is zero. If
\[
\mu_1\ge\cdots\ge\mu_m\ge0
\]
are the eigenvalues of $M[v]$, then
\[
-2\lambda\mu_1\le\cdots\le-2\lambda\mu_m
\]
are those of $A[v]$. Hence $\mu_r=0$ and
\begin{equation}\label{eq:rankbound}
\rank M[v]\le r-1<m.
\end{equation}

Choose a ball $V'\Subset V$ and
$v_0\in C_c^\infty(V)$ with $v_0\ge0$ and $v_0>0$ on $V'$. By
\eqref{eq:rankbound}, $M[v_0]$ is singular. Choose
$0\ne c\in\ker M[v_0]$ and set
$\psi=\sum_i c_i\phi^{(i)}$. Then
\[
0=c^{\!\top}M[v_0]c
=\int_{\Sph^2}v_0\,\psi^2e^{2u}\dd.
\]
Since $v_0>0$ and $e^{2u}>0$ on $V'$, we conclude that
$\psi\equiv0$ on $V'$.
\end{proof}
\begin{proof}[Proof of Theorem~\ref{thm:intro-smooth}]
Suppose, for contradiction, that $K_\star\not\equiv1$. By Lemmas~\ref{lem:slack}
and~\ref{lem:relation}, there exist a ball $V'\subset\Sph^2$ and
$0\ne c\in\R^m$ such that
\[
\psi:=\sum_{i=1}^m c_i\phi^{(i)}\equiv0\quad\text{on }V'.
\]
Since each $\phi^{(i)}$ satisfies
\[
\Delta_0\phi^{(i)}=\lambda e^{2u}\phi^{(i)},
\]
we have
\begin{equation}\label{eq:psieq}
\Delta_0\psi=\lambda e^{2u}\psi
\quad\text{on }\Sph^2.
\end{equation}
By Aronszajn's strong unique continuation theorem \cite{Aronszajn}, the
vanishing of $\psi$ on the nonempty open set $V'$ implies
$\psi\equiv0$ on the connected sphere. But the $\phi^{(i)}$ are
$L^2(e^{2u}\dd)$-orthonormal, so
\[
0=\int_{\Sph^2}\psi^2e^{2u}\dd
=\sum_{i=1}^m c_i^2
=|c|^2,
\]
contradicting $c\ne0$. Hence $K_\star\equiv1$. A smooth metric on
$\Sph^2$ with $K\equiv1$ is isometric to $(\Sph^2,\round)$, and therefore
\[
\inf_{\Cc}\lambda_k
=
\lambda_k(g_\star)
=
\lambda_k(\Sph^2).
\]
\end{proof}

\section{Spectral rigidity in the completion}\label{sec:completion-rigidity}

\begin{proposition}\label{prop:weyl}
For $g\in\overline\Cc$, the eigenvalue counting function
$N_g(x)=\#\{k\ge1:\lambda_k(g)\le x\}$ satisfies
\begin{equation}\label{eq:weyl}
N_g(x)=\frac{\Area(g)}{4\pi}\,x+o(x)
\qquad(x\to\infty).
\end{equation}
\end{proposition}

\begin{proof}
Every $g\in\overline\Cc$, with its area measure
$\mathfrak m=\mathcal H^2$, is a compact two-dimensional Alexandrov space
of curvature $\ge1$, hence Ahlfors $2$-regular. Indeed, the upper bound
\[
\mathcal H^2(B_r(x))\le\pi r^2
\]
follows from Bishop's inequality, while the corresponding lower bound
follows from Bishop--Gromov. Thus the hypotheses of
\cite[Corollary~4.8]{AHT} apply and give
\[
\lim_{x\to\infty}\frac{N_g(x)}{x}
=
\frac{\omega_2}{(2\pi)^2}\mathcal H^2(X)
=
\frac{\Area(g)}{4\pi},
\]
which is \eqref{eq:weyl}.
\end{proof}

\begin{proof}[Proof of Proposition~\ref{prop:rigidity}]
Since $\lambda_{\ell^2}(g)=\ell(\ell+1)$ and the eigenvalues are
nondecreasing,
\[
N_g(\ell(\ell+1))\ge\ell^2.
\]
Moreover,
\[
\lambda_{(\ell+1)^2}(g)=(\ell+1)(\ell+2)>\ell(\ell+1),
\]
so
\[
N_g(\ell(\ell+1))\le(\ell+1)^2-1.
\]
Hence
\begin{equation}\label{eq:squeeze}
\ell^2\le N_g\bigl(\ell(\ell+1)\bigr)
\le(\ell+1)^2-1.
\end{equation}
Dividing by $\ell(\ell+1)$ and letting $\ell\to\infty$ gives
\[
\frac{N_g(\ell(\ell+1))}{\ell(\ell+1)}\longrightarrow1.
\]
Comparing with \eqref{eq:weyl} yields
\[
\Area(g)=4\pi.
\]

Let $(X,d)$ be the Alexandrov surface underlying $g$. Since
$(X,d,\mathcal H^2)$ is $\mathrm{CD}(1,2)$ by
Proposition~\ref{prop:rcd}, Bishop--Gromov gives, for every $p\in X$,
\[
\mathcal V_p(r):=
\frac{\mathcal H^2(B_r(p))}
{2\pi(1-\cos r)}
\]
nonincreasing on $(0,\pi]$. At every point, the tangent cone is a Euclidean
cone of total angle $\theta(p)\le2\pi$, and hence
\[
\lim_{r\to0^+}\mathcal V_p(r)
=
\frac{\theta(p)}{2\pi}\le1.
\]
If $\diam(X)<\pi$, choose $r\in(\diam(X),\pi)$. Then $B_r(p)=X$, and therefore
\[
\mathcal V_p(r)
=
\frac{4\pi}{2\pi(1-\cos r)}
>1,
\]
contradicting the monotonicity of $\mathcal V_p$. Thus
\[
\diam(X)=\pi.
\]

By maximal-diameter rigidity for Alexandrov spaces of curvature $\ge1$
\cite[Theorem~1.8]{ZhangZhu}, $X$ is a spherical suspension over a compact
one-dimensional Alexandrov space of curvature $\ge1$. The cross-section
cannot be a segment since $X$ is a closed surface, so it is a circle
$S^1(L)$ with $L\le2\pi$. Away from the poles the suspension has metric
\[
dt^2+\sin^2t\,d\vartheta^2,
\qquad t\in(0,\pi),\quad \vartheta\in\R/L\Z,
\]
and hence
\[
\Area(g)=\mathcal H^2(X)
=\int_0^\pi L\sin t\,dt
=2L.
\]
Since $\Area(g)=4\pi$, we obtain $L=2\pi$, and therefore $X$ is isometric
to the round unit sphere. Since the reference measure of $g$ is
$\mathcal H^2$, $g$ is the round unit sphere.
\end{proof}

\begin{remark}\label{rem:rigidity-extras}
The suspension over $S^1(2\pi\alpha)$ is the spindle
$\Sph^2_\alpha$ of \eqref{eq:spindle}, with area $4\pi\alpha$; thus
$L=2\pi$ distinguishes the round sphere within this family. Proposition
\ref{prop:rigidity} requires equality at every cluster bottom and therefore
does not address minimizers of a single $f_k$. Conjecture~\ref{conjecture}
predicts that a singular such minimizer, if it exists, carries its excess
curvature in $\omega_{\mathrm{sing}}$, which the present argument does not
exclude.
\end{remark}

\section{A sharp Sturm--Liouville comparison theorem}
\label{sec:1d-comparison}

This section studies the problem \eqref{eq:slp}. Both \eqref{eq:concavity} and
$\mathsf L_\varphi$ are unchanged when $\varphi$ is multiplied by a positive constant,
so throughout this section we normalize
\begin{equation}\label{eq:normalization}
\max_{[0,D]}\varphi=1.
\end{equation}

\subsection{The Pr\"ufer phase}\label{subsec:phase}

Fix $\mu>0$ and let $u$ be a nontrivial solution of
\[
(\varphi u')'+\mu\varphi u=0
\]
satisfying the natural condition at $0$. Set $v=\varphi u'$. By
Lemma~\ref{lem:endpoints}, $(u,v)\to(c_0,0)$ as $t\to0^+$ with $c_0\ne0$
when $\varphi(0)=0$, while the same conclusion follows from regularity when
$\varphi(0)>0$. Replacing $u$ by $-u$, assume $c_0>0$. Since
\[
u'=\varphi^{-1}v,\qquad v'=-\mu\varphi u,
\]
the pair $(u,v)$ never vanishes simultaneously. We therefore define its
Pr\"ufer phase $\gamma=\gamma(\,\cdot\,;\mu)$ by
\[
(u,v)=R(\cos\gamma,-\sin\gamma),\qquad
R=(u^2+v^2)^{1/2},\qquad \gamma(0^+)=0.
\]

\begin{lemma}\label{lem:phase-basic}
The phase $\gamma$ is $C^1$ and satisfies
\begin{equation}\label{eq:phase-ode}
\gamma'
=
\mu\varphi\cos^2\gamma+\varphi^{-1}\sin^2\gamma>0.
\end{equation}
Thus $\gamma$ is strictly increasing, and the zeros of $u$ are precisely
the points where $\gamma\in\frac{\pi}{2}+\pi\mathbb Z$. Moreover,
$\gamma(D^-)$ exists and is finite; if $u$ satisfies the natural condition at
$D$, then $\gamma(D^-)\in\pi\mathbb Z$.
\end{lemma}

\begin{proof}
Differentiating $\tan\gamma=-v/u$ and using the first-order system gives
\eqref{eq:phase-ode}. The positivity follows from $\varphi>0$ and the fact
that $u$ and $v$ cannot vanish simultaneously. Hence $\gamma$ is increasing,
and $u=R\cos\gamma$ gives the zero characterization. If
$\gamma(D^-)=\infty$, then $u$ has infinitely many zeros near $D$, contrary
to the non-oscillation of the endpoint. Finally, under the natural condition
at $D$, Lemma~\ref{lem:endpoints} gives $u(D^-)\ne0$ and
$v(D^-)=0$, so $\gamma(D^-)\in\pi\mathbb Z$.
\end{proof}

\begin{lemma}\label{lem:phase-eigen}
Let $j\ge1$ and let $\gamma$ be the phase of the $j$-th eigenfunction
$u_j$, with $\mu=\mu_j(\varphi)$. Then
\[
\gamma(0^+)=0,\qquad \gamma(D^-)=j\pi.
\]
\end{lemma}

\begin{proof}
By Lemma~\ref{lem:phase-basic}, $\gamma(D^-)=m\pi$ for some $m\ge0$.
Since $\gamma$ is strictly increasing, it crosses
$\frac{\pi}{2}+\pi\mathbb Z$ exactly $m$ times. These crossings are precisely
the zeros of $u_j$, so $m=j$.
\end{proof}

We write $\bar\gamma(\theta;\mu)$ for the phase of the natural-at-$0$
solution of the Legendre equation with density $\sin\theta$, normalized by
$\bar\gamma(0^+;\mu)=0$.

\begin{lemma}\label{lem:phase-asymptotic}
Suppose $\varphi(0)=0$ and fix $\mu>0$. With
\[
\Phi(t):=\int_0^t\varphi(s)\,ds,
\]
the phase satisfies
\[
\gamma(t)=\mu\Phi(t)(1+o(1))
\qquad\text{as }t\to0^+.
\]
In particular, for the model density $\sin\theta$,
\begin{equation}\label{eq:model-asymptotic}
\bar\gamma(\theta;\mu)
=
\mu(1-\cos\theta)(1+o(1))
=
\frac{\mu}{2}\theta^2(1+o(1)).
\end{equation}
\end{lemma}

\begin{proof}
The natural condition at $0$ gives
\[
v(t)=-\mu\int_0^t\varphi(s)u(s)\,ds.
\]
Since $u(t)\to c_0>0$ by Lemma~\ref{lem:endpoints},
\[
v(t)=-\mu c_0\Phi(t)(1+o(1)).
\]
Hence
\[
-\frac{v(t)}{u(t)}
=
\mu\Phi(t)(1+o(1)).
\]
As $\gamma\to0$ and $\tan\gamma=-v/u$, this yields
$\gamma=\mu\Phi(1+o(1))$. For $\varphi(\theta)=\sin\theta$,
$\Phi(\theta)=1-\cos\theta$, giving \eqref{eq:model-asymptotic}.
\end{proof}

\subsection{Reparametrization by the density}\label{subsec:clock}

From here to the end of \S\ref{subsec:sharp} we assume
\eqref{eq:intro-density}. This condition is preserved by the reflection
$\psi(s)=\varphi(D-s)$ and implies strict concavity. Hence $\varphi$
has a unique maximum at some $t_*\in[0,D]$, is strictly increasing on
$[0,t_*]$, and strictly decreasing on $[t_*,D]$.

\begin{lemma}\label{lem:energy}
Let $\varphi'_+$ denote the right derivative of $\varphi$ and set
\[
E:=\varphi^2+(\varphi'_+)^2.
\]
Then
\[
E\ge1\quad\text{on }(0,D),
\]
with $E$ nonincreasing on $(0,t_*)$ and nondecreasing on $(t_*,D)$.
\end{lemma}

\begin{proof}
On intervals where $\varphi$ is $C^2$,
\[
E'=2\varphi'(\varphi''+\varphi),
\]
so $E'\le0$ where $\varphi'>0$ and $E'\ge0$ where $\varphi'<0$.
For general concave $\varphi$ the same conclusion holds distributionally. Fix a compact
subinterval $[a,b]\subset(0,t_*)$. There $\varphi'$ is bounded and of bounded variation,
and $\varphi'\ge\varphi'(b)>0$: indeed $\varphi'\ge0$ on $(0,t_*)$ because $\varphi$ is
nondecreasing there, and a zero of $\varphi'$ at some $s<t_*$ would force $\varphi'<0$
on $(s,t_*)$ by strict concavity, contradicting that $\varphi$ increases up to $t_*$. By
the Leibniz rule for $\mathrm{BV}$ functions \cite[Thm.~3.99]{AmbrosioFuscoPallara}, the
distributional derivative of $(\varphi')^2$ on $(a,b)$ is $2\bar\varphi'\,d\varphi'$,
where $\bar\varphi'=\tfrac12\bigl(\varphi'(\cdot^+)+\varphi'(\cdot^-)\bigr)$ is the
symmetric average, equal to $\varphi'$ off the countable jump set. Since $\varphi$ is
locally Lipschitz on $(0,D)$, this gives
\[
dE=2\varphi\varphi'\,dt+2\bar\varphi'\,d\varphi'=2\bar\varphi'\,(d\varphi'+\varphi\,dt)\le0
\qquad\text{on }(a,b),
\]
because $\bar\varphi'>0$ and \eqref{eq:intro-density} gives $d\varphi'+\varphi\,dt\le0$.
As $[a,b]$ was arbitrary and $E$ is right-continuous, $E$ is nonincreasing on
$(0,t_*)$. The argument on $(t_*,D)$ is identical with the inequalities reversed. Since
$\varphi(t_*)=1$,
\[
E(t)\ge1
\]
on both sides of $t_*$, and the same holds at $t_*$ directly.
\end{proof}

\begin{lemma}\label{lem:clock}
Define $\tau\colon[0,D]\to[0,\pi]$ by
\[
\tau(t)=
\begin{cases}
\arcsin\varphi(t),&t\in[0,t_*],\\[2pt]
\pi-\arcsin\varphi(t),&t\in[t_*,D].
\end{cases}
\]
Then $\tau$ is absolutely continuous and strictly increasing,
\[
\sin\tau(t)=\varphi(t),
\qquad
\tau'(t)
=
\frac{|\varphi'(t)|}{\sqrt{1-\varphi(t)^2}}
\ge1
\quad\text{a.e.}
\]
Consequently,
\[
D\le\tau(D)-\tau(0)\le\pi.
\]
\end{lemma}

\begin{proof}
Strict concavity gives the stated monotonicity and hence continuity and strict
increase of $\tau$. On a compact subinterval of $(0,D)\setminus\{t_*\}$ one has
$\varphi\le1-\eta$ for some $\eta>0$, while $\varphi$ is locally Lipschitz on $(0,D)$;
hence $\tau$ is locally Lipschitz there, and the chain rule gives
\[
\tau'=\frac{|\varphi'|}{\sqrt{1-\varphi^2}}
\]
a.e., the inequality $\tau'\ge1$ being Lemma~\ref{lem:energy}.

For absolute continuity, take one of the two monotonicity intervals, say $[0,t_*]$. By
the preceding paragraph $\tau$ is absolutely continuous on $[a,b]$ whenever
$0<a<b<t_*$, so $\tau(b)-\tau(a)=\int_a^b\tau'$; letting $a\downarrow0$, $b\uparrow t_*$
and using the continuity of $\tau$ together with monotone convergence gives
$\int_0^{t_*}\tau'=\tau(t_*)-\tau(0)$. For a continuous nondecreasing function one
always has $\int_\alpha^\beta\tau'\le\tau(\beta)-\tau(\alpha)$, so equality across
$[0,t_*]$ forces equality on every subinterval, which is absolute continuity. The same
argument applies on $[t_*,D]$, hence on $[0,D]$. Lipschitz continuity would be false: if
$\varphi$ has a corner at $t_*$ then $1-\varphi^2$ vanishes linearly while
$|\varphi'|\to\varphi'(t_*^-)>0$, so $\tau'$ blows up like $(t_*-t)^{-1/2}$, integrably.

Consequently $\tau(D)-\tau(0)=\int_0^D\tau'\,dt\ge D$, and
$\tau([0,D])\subset[0,\pi]$ yields $D\le\pi$.
\end{proof}

\subsection{Comparison with the Legendre equation}\label{subsec:phasecomp}

Transporting the model phase by $\tau$ enlarges both coefficients in
\eqref{eq:phase-ode} by the factor $\tau'\ge1$.

\begin{lemma}\label{lem:supersol}
Fix $\mu>0$ and set
\[
\zeta(t):=\bar\gamma(\tau(t);\mu).
\]
Then $\zeta$ is continuous and locally absolutely continuous on $(0,D)$ and
satisfies, a.e.,
\begin{equation}\label{eq:supersol}
\zeta'
=
g_2\cos^2\zeta+r_2\sin^2\zeta,
\qquad
r_2:=\frac{\tau'}{\varphi},\quad
g_2:=\mu\varphi\tau',
\end{equation}
with
\begin{equation}\label{eq:coeff-domination}
r_2\ge\varphi^{-1},
\qquad
g_2\ge\mu\varphi.
\end{equation}
\end{lemma}

\begin{proof}
By Lemma~\ref{lem:clock}, $\tau$ is absolutely continuous and
$\sin\tau=\varphi$. Hence the chain rule and the model phase equation give
\[
\zeta'
=
\tau'\left(
\mu\sin\tau\cos^2\zeta+
(\sin\tau)^{-1}\sin^2\zeta
\right),
\]
which is \eqref{eq:supersol}. The inequalities
\eqref{eq:coeff-domination} follow from $\tau'\ge1$.
\end{proof}

\begin{lemma}\label{lem:comparison}
Fix $\mu>0$, let $\gamma$ be the phase of the natural-at-$0$ solution of
$(\varphi u')'+\mu\varphi u=0$, and let $\zeta$ be as above. Then
\begin{equation}\label{eq:phase-comparison}
\gamma(t)\le\zeta(t)=\bar\gamma(\tau(t);\mu)
\qquad\text{for all }t\in(0,D).
\end{equation}
\end{lemma}

\begin{proof}
On every compact subinterval of $(0,D)$, the coefficient pair
\[
(r_1,g_1)=(\varphi^{-1},\mu\varphi)
\]
is dominated by $(r_2,g_2)$ in \eqref{eq:coeff-domination}. Hence
\cite[Thm.~4.5.2(1)]{Zettl} propagates $\zeta\ge\gamma$ from any point
$t_0>0$ to the right. It remains to seed the comparison near $0$.

If $\varphi(0)>0$, then $\tau(0)>0$, so $\zeta(0)>0=\gamma(0^+)$ and the
inequality holds for all sufficiently small $t$.

Suppose $\varphi(0)=0$. By Lemma~\ref{lem:energy},
$\varphi'(0^+)\ge1$. If $\kappa:=\varphi'(0^+)>1$, then
Lemma~\ref{lem:phase-asymptotic} and $\sin\tau=\varphi$ give
\[
\gamma(t)=\mu\Phi(t)(1+o(1)),
\qquad
\zeta(t)=\frac{\mu}{2}\varphi(t)^2(1+o(1)).
\]
Since
\[
\frac12(\varphi^2)'=\varphi\varphi'
\ge(\kappa-o(1))\varphi,
\]
we obtain
\[
\liminf_{t\to0^+}\frac{\zeta(t)}{\gamma(t)}\ge\kappa>1.
\]
Thus $\zeta>\gamma$ near $0$.

Finally, if $\varphi'(0^+)=1$, then Lemma~\ref{lem:energy} gives
$E\equiv1$ on $(0,t_*)$, hence
\[
\varphi'=\sqrt{1-\varphi^2}\quad\text{a.e. on }(0,t_*).
\]
Therefore $\varphi(t)=\sin t$ and $\tau(t)=t$ on $[0,t_*]$. The two phases
then solve the same problem with the same normalization at $0$, so
$\gamma\equiv\zeta$ on $[0,t_*]$. The comparison on $(0,D)$ follows by
the first paragraph.
\end{proof}

A comparison at the right endpoint alone loses one eigenvalue index; reflection
and an interior comparison remove this loss.

\begin{lemma}\label{lem:reflection}
Let $j\ge1$, let $u_j$ be the $j$-th eigenfunction, and put
$\mu=\mu_j(\varphi)$. Set
\[
\psi(s):=\varphi(D-s),\qquad
\hat\tau(s):=\pi-\tau(D-s),\qquad
\hat\gamma(s):=j\pi-\gamma(D-s).
\]
Then $\hat\gamma$ is the phase of the reflected natural-at-$0$ solution and
$\hat\tau$ is the density clock for $\psi$. Consequently,
\begin{equation}\label{eq:reflected-comparison}
\hat\gamma(s)
\le
\bar\gamma(\hat\tau(s);\mu_j)
\qquad\text{for all }s\in(0,D).
\end{equation}
\end{lemma}

\begin{proof}
By reflection invariance, $\psi$ satisfies the same hypotheses as $\varphi$.
Using $\gamma(D^-)=j\pi$ from Lemma~\ref{lem:phase-eigen}, direct differentiation
gives
\[
\hat\gamma'
=
\mu\psi\cos^2\hat\gamma+\psi^{-1}\sin^2\hat\gamma,
\qquad
\hat\gamma(0^+)=0.
\]
Thus $\hat\gamma$ is the normalized phase for the reflected eigenfunction.
Similarly,
\[
\sin\hat\tau=\psi,\qquad \hat\tau'=\tau'(D-s)\ge1.
\]
Applying Lemma~\ref{lem:comparison} to $\psi$ gives
\eqref{eq:reflected-comparison}.
\end{proof}

\begin{lemma}\label{lem:model}
For each fixed $\theta\in(0,\pi)$, the map
$\mu\mapsto\bar\gamma(\theta;\mu)$ is strictly increasing. Moreover, for
every $j\ge1$,
\begin{equation}\label{eq:model-identity}
\bar\gamma(\theta;j(j+1))
+
\bar\gamma(\pi-\theta;j(j+1))
=
j\pi.
\end{equation}
In particular, if $\mu<j(j+1)$, then
\[
\bar\gamma(\theta;\mu)
+
\bar\gamma(\pi-\theta;\mu)
<j\pi.
\]
\end{lemma}

\begin{proof}
For $0<\mu_1<\mu_2$, the model phases have the same coefficient
$r=(\sin\theta)^{-1}$ and satisfy
\[
g_1=\mu_1\sin\theta<\mu_2\sin\theta=g_2.
\]
By \eqref{eq:model-asymptotic}, the phase for $\mu_2$ is strictly larger near
$0$, and \cite[Thm.~4.5.2(5)]{Zettl} then gives
\[
\bar\gamma(\theta;\mu_2)>
\bar\gamma(\theta;\mu_1)
\qquad (0<\theta<\pi).
\]

For $\mu=j(j+1)$, the natural eigenfunction is
$P_j(\cos\theta)$, and Lemma~\ref{lem:phase-eigen} gives
$\bar\gamma(\pi^-)=j\pi$. Since the model is symmetric, the reflected phase
$j\pi-\bar\gamma(\pi-\theta)$ is the normalized phase of the same problem.
Uniqueness of the natural solution up to scale gives
\eqref{eq:model-identity}. The final assertion follows from strict
monotonicity.
\end{proof}
\subsection{The sharp comparison}\label{subsec:sharp}

\begin{lemma}\label{lem:splitting}
Let $j\ge1$, put $\mu=\mu_j(\varphi)$, and let $t_0\in(0,D)$ with
$\theta_0:=\tau(t_0)$. Then
\begin{equation}\label{eq:sandwich}
j\pi
=
\gamma(t_0)+\hat\gamma(D-t_0)
\le
\bar\gamma(\theta_0;\mu)
+
\bar\gamma(\pi-\theta_0;\mu).
\end{equation}
\end{lemma}

\begin{proof}
The equality follows from the definition
$\hat\gamma(s)=j\pi-\gamma(D-s)$, while the inequality follows from
Lemmas~\ref{lem:comparison} and~\ref{lem:reflection}, together with
$\hat\tau(D-t_0)=\pi-\tau(t_0)$.
\end{proof}

\begin{theorem}\label{thm:sl}
Assume \eqref{eq:intro-density} and \eqref{eq:normalization}. Then
\begin{equation}\label{eq:sl-ineq}
\mu_j(\varphi)\ge j(j+1)\qquad\text{for every }j\ge0.
\end{equation}
For $j\ge1$, equality holds if and only if
\[
D=\pi,\qquad \varphi(t)=\sin t,
\]
in which case equality holds for every $j$.
\end{theorem}

\begin{proof}
\textbf{Inequality.} The case $j=0$ is immediate. For $j\ge1$, put
$\mu=\mu_j(\varphi)$. If $\mu<j(j+1)$, then Lemma~\ref{lem:model} gives
\[
\bar\gamma(\theta_0;\mu)
+
\bar\gamma(\pi-\theta_0;\mu)
<j\pi,
\]
contradicting \eqref{eq:sandwich}.

\textbf{Equality.}
The model $\varphi=\sin t$ on $[0,\pi]$ has eigenvalues $j(j+1)$, so the
converse is immediate. Suppose
\[
\mu_j(\varphi)=j(j+1)=:\mu
\]
for some $j\ge1$. By \eqref{eq:model-identity}, equality holds in
\eqref{eq:sandwich} for every $t_0$. Since each term on the left is bounded
by the corresponding term on the right, Lemmas~\ref{lem:comparison} and
\ref{lem:reflection} give
\[
\gamma(t)=\bar\gamma(\tau(t);\mu)
\qquad\text{for all }t\in(0,D).
\]
The reflected comparison gives the analogous equality from the right.

If $\varphi(0)>0$, Case~1 of Lemma~\ref{lem:comparison} gives
$\bar\gamma(\tau(0);\mu)>0=\gamma(0^+)$, a contradiction. Hence
$\varphi(0)=0$. Applying the same argument to the reflected density gives
$\varphi(D)=0$.

Now \eqref{eq:phase-ode} and \eqref{eq:supersol} imply
\[
(\tau'-1)
\bigl(\mu\varphi\cos^2\gamma+\varphi^{-1}\sin^2\gamma\bigr)=0
\quad\text{a.e.}
\]
The bracket is positive, so $\tau'=1$ a.e. Thus
\[
(\varphi')^2=1-\varphi^2
\quad\text{a.e. on }(0,D).
\]
Together with Lemma~\ref{lem:energy}, this gives $E\equiv1$. Since
$\varphi'$ is monotone, its only possible discontinuities are jumps. At a
jump point $s$, the preceding identity gives
\[
|\varphi'(s^-)|=|\varphi'(s^+)|.
\]
Monotonicity would therefore force a sign change, which can occur only at
$t_*$. But $\varphi(t_*)=1$, so both one-sided limits there are zero.
Thus $\varphi'$ is continuous, and
\[
\varphi''+\varphi=0
\qquad\text{on }(0,t_*)\cup(t_*,D).
\]
Since $\varphi(t_*)=1$ and $\varphi'(t_*)=0$,
\[
\varphi(t)=\cos(t-t_*).
\]
The endpoint conditions $\varphi(0)=\varphi(D)=0$ and
$\varphi>0$ on $(0,D)$ then give
\[
t_*=\frac{\pi}{2},\qquad D=\pi,\qquad \varphi(t)=\sin t.
\]
\end{proof}

The one-dimensional $\mathrm{RCD}(1,2)$ spaces are classified in
\cite{KitabeppuLakzian} as circles or intervals with density $e^{-f}$,
where $f$ is $(1,2)$-convex. The positive curvature lower bound excludes
the circle \cite[\S4]{KitabeppuLakzian}, so such a space is an interval
$[0,D]$ with $\varphi=e^{-f}$ satisfying \eqref{eq:intro-density}. By
Lemma~\ref{lem:cheeger-identification}, its Neumann Laplacian is
$\mathsf L_\varphi$, and the preceding theorem applies.

\begin{proof}[Proof of Theorem~\ref{thm:intro-1d}]
Multiplying $\varphi$ by a positive constant changes neither
\eqref{eq:intro-density} nor $\mathsf L_\varphi$. Thus we may impose
\eqref{eq:normalization}. The result is then exactly Theorem~\ref{thm:sl}.
\end{proof}

\begin{remark}\label{rem:sharp}
By Corollary~\ref{cor:spindle-collapse}, the model is, up to a constant
factor in its measure, the mGH limit of the collapsing spindles
$\Sph^2_\alpha$, whose radial spectra are exactly $\{j(j+1)\}$. Hence
\eqref{eq:main-ineq} is attained on the one-dimensional stratum of the mGH
completion of $\Cc$, and cannot be improved.
\end{remark}

\appendix
\section{Proof of Theorem~\ref{thm:slp_discrete}}\label{Appendix}
Consider the Sturm Liouville Problem \eqref{eq:slp}
\begin{equation*}
-(\varphi u')'=\mu\,\varphi u \quad\text{on }(0,D),
\qquad
\lim_{t\to0^+}\varphi u'=\lim_{t\to D^-}\varphi u'=0.
\end{equation*}
The associated energy form is $\mathcal E(u)=\int_0^D (u')^2\varphi\,dt$,
with form domain
$H^1(\varphi)=\{u\in L^2(\varphi\,dt)\cap AC_{\mathrm{loc}}(0,D): u'\in L^2(\varphi\,dt)\}$,
and maximal domain
\[
D_{\max}=\bigl\{u\in L^2(\varphi\,dt):\ u,\ \varphi u'\in AC_{\mathrm{loc}}(0,D),\
-\varphi^{-1}(\varphi u')'\in L^2(\varphi\,dt)\bigr\}.
\]
\begin{remark}\label{rem:reflection-invariance}
The hypotheses are invariant under the reflection $\psi(s)=\varphi(D-s)$, so it suffices
to prove endpoint statements at $0$.
\end{remark}

\subsection{The Endpoint Classification}\label{subsec:endpoints}

At an endpoint $e\in\{0,D\}$ with $\varphi(e)>0$, the problem is regular and
the natural condition is the classical Neumann condition $u'(e)=0$. If $\varphi(e)=0$ the endpoint is
regular when $\int_e\varphi^{-1}\,dt<\infty$ and singular otherwise. The following lemma treats
both cases simultaneously.  

\begin{lemma}\label{lem:endpoints}
Let $e\in\{0,D\}$ with $\varphi(e)=0$ and fix $\mu\in\R$, and consider
$(\varphi u')'+\mu\varphi u=0$ in the quasi-derivative sense
$u,\varphi u'\in AC_{\mathrm{loc}}(0,D)$. Then $e$ is limit-circle non-oscillatory
\textup{(LCNO)} in the sense of \cite[\S7.3]{Zettl}; the limit
$\lim_{t\to e}\varphi(t)u'(t)$ exists for every solution; and the solutions for which it
vanishes form a one-dimensional space, each nontrivial member having a finite nonzero
limit at $e$. Finally, the operator $\mathsf L_\varphi$
associated with the closed form $(\mathcal E,H^1(\varphi))$ is self-adjoint with the
separated domain
\[
D(\mathsf L_\varphi)=\bigl\{u\in D_{\max}:\ \lim_{t\to e}\varphi(t)u'(t)=0
\ \text{ at each endpoint } e\in\{0,D\}\bigr\}.
\]
\end{lemma}

\begin{proof}
Say $e=0$; the case $e=D$ follows by Remark~\ref{rem:reflection-invariance}. Fix
$t_0\in(0,D)$ and put $c:=\varphi(t_0)/t_0>0$. Concavity together with $\varphi(0)=0$
gives the linear lower bound $\varphi(t)\ge ct$ on $[0,t_0]$, while $\varphi\le1$ by
\eqref{eq:normalization}.

 \emph{Step 1: Classification of the endpoint.} 
 For $\mu=0$ in \eqref{eq:slp}, the solutions are
\[
u_1\equiv1,
\qquad
u_2(t)=\int_t^{t_0}\varphi(s)^{-1}\,ds.
\]
Since $0\le\varphi\le1$,
\[
\int_0^{t_0}\varphi(t)\,dt\le t_0<\infty,
\]
so $u_1\in L^2((0,t_0),\varphi\,dt)$. Moreover, $\varphi(t)\ge ct$ gives
\[
\int_x^{t_0}\varphi(t)^{-1}\,dt
\le c^{-1}\log(t_0/x),
\]
and hence
\[
\int_0^{t_0}\varphi(x)
\Bigl(\int_x^{t_0}\varphi(t)^{-1}\,dt\Bigr)^2dx
\le
c^{-2}\int_0^{t_0}\log^2(t_0/x)\,dx<\infty.
\]
Thus both solutions belong to $L^2(\varphi\,dt)$ near $0$. Since the
classification of an endpoint as limit-point or limit-circle is independent
of the spectral parameter \cite[\S7.2]{Zettl}, the endpoint $0$ is
limit-circle. Likewise, the oscillatory or non-oscillatory character of a
limit-circle endpoint is independent of the spectral parameter
\cite[Thm.~7.3.1]{Zettl}. At $\mu=0$, every nonzero solution
\[
\alpha+\beta\int_t^{t_0}\varphi^{-1}
\]
is monotone and hence has at most one zero in $(0,t_0)$. Thus $0$ is
limit-circle non-oscillatory (LCNO).

\emph{Step 2: limits at the endpoint.} 
Let $u$ be any solution. By Step 1 and the Cauchy--Schwarz inequality,
\[
\int_0^t\varphi|u|
\le
\left(\int_0^t\varphi\right)^{1/2}
\left(\int_0^t\varphi u^2\right)^{1/2}
<\infty.
\]
Hence $(\varphi u')'=-\mu\varphi u\in L^1$ near $0$, and therefore
$\lim_{t\to0^+}\varphi(t)u'(t)$
exists and is finite. Suppose now that $\lim_{t\to0^+}\varphi(t)u'(t)=0$. Then
\[
\varphi(t)u'(t)=-\mu\int_0^t\varphi u,
\]
so, using $\varphi(t)\ge ct$,
\[
|u'(t)|
\le
\frac{|\mu|}{ct}\int_0^t\varphi|u|
=o(1/t).
\]
Thus $u(t)=O(\log(1/t))$. Since $\varphi\le1$, substituting this back gives
\[
|u'(t)|
\le
\frac{C}{t}\int_0^t\log(1/s)\,ds
=O(\log(1/t)),
\]
which is integrable near $0$. Hence
$c_0=\lim_{t\to0^+}u(t)$
exists and is finite. Finally, if $c_0=0$, then for sufficiently small $t>0$,
\[
M(t)=\sup_{0<s\le t}|u(s)|<\infty.
\]
For $s\le t$,
\[
\int_0^s\varphi|u|\le sM(t),
\]
and hence
\[
|u'(s)|\le\frac{|\mu|}{c}M(t).
\]
If $\mu=0$, the natural condition gives $\varphi u'\equiv0$, so $u$ is
constant and $u(0^+)=0$ implies $u\equiv0$. If $\mu\ne0$, integration from
$0$ gives
\[
M(t)\le\frac{|\mu|t}{c}M(t),
\]
which forces $M(t)=0$ for $t<c/|\mu|$. Thus in either case $u\equiv0$ near
$0$ (and hence, by uniqueness, on $(0,D)$).

\emph{Step 3: the natural solutions.}
By Step 2, the map
\[
u\longmapsto\lim_{t\to0^+}\varphi(t)u'(t)
\]
is a well-defined linear functional on the two-dimensional solution space. It is not
identically zero: otherwise every solution would have a finite limit at $0$ by
Step 2, and for two linearly independent solutions $u_1,u_2$ the Wronskian
\[
u_1(\varphi u_2')-u_2(\varphi u_1')
\]
would tend to $0$, contradicting that it is a nonzero constant. Hence its kernel
is one-dimensional. By Step 2, every nontrivial solution in this kernel has a
finite nonzero limit at $0$.

\emph{Step 4: the operator and its domain.}
The form $\mathcal E$ is closed: if $(u_n)$ is Cauchy for
$\mathcal E+\|\cdot\|_{L^2(\varphi\,dt)}^2$, then $u_n\to u$ and $u_n'\to g$ in
$L^2(\varphi\,dt)$, and local $H^1$ convergence gives $g=u'$. Hence, by the
representation theorem for closed nonnegative forms, the associated operator
$\mathsf L_\varphi$ is self-adjoint, and testing against $C_c^\infty(0,D)$ gives
\[
\mathsf L_\varphi u=-\varphi^{-1}(\varphi u')',
\qquad
D(\mathsf L_\varphi)\subset D_{\max}.
\]

For $u\in D_{\max}$, Cauchy--Schwarz gives
\[
\int_0^D|(\varphi u')'|
\le
\left(\int_0^D\varphi\right)^{1/2}
\left\|\varphi^{-1}(\varphi u')'\right\|_{L^2(\varphi\,dt)}
<\infty.
\]
Thus $\varphi u'$ has finite limits at both endpoints. Moreover, using
$\varphi(t)\ge ct$ and $\varphi\le1$,
\[
\varphi(t)u'(t)-\varphi(0^+)u'(0^+)=o(\sqrt t),
\qquad
u(t)=O(\log(1/t))
\quad\text{as }t\to0^+,
\]
and similarly at $D$.

Let $u\in D(\mathsf L_\varphi)$ and choose a cutoff $\chi\in H^1(\varphi)$
which is $1$ near $0$ and vanishes away from $0$. Integrating by parts and
using the defining identity for the associated operator gives
\[
\int_0^D(\mathsf L_\varphi u)\chi\,\varphi\,dt
=
\mathcal E(u,\chi)
+
\lim_{t\to0^+}\varphi(t)u'(t).
\]
Since the left-hand side equals $\mathcal E(u,\chi)$, the boundary term
vanishes. The same argument at $D$ yields
\[
D(\mathsf L_\varphi)\subseteq D(\mathsf L_N),
\]
where $\mathsf L_N=-\varphi^{-1}(\varphi u')'$ with
\[
\lim_{t\to0^+}\varphi(t)u'(t)
=
\lim_{t\to D^-}\varphi(t)u'(t)
=0.
\]
Conversely, Green's formula and the endpoint estimates above show that
$\mathsf L_N$ is symmetric. Hence
\[
\mathsf L_\varphi
\subseteq
\mathsf L_N
\subseteq
\mathsf L_N^*
\subseteq
\mathsf L_\varphi^*
=
\mathsf L_\varphi,
\]
so $\mathsf L_\varphi=\mathsf L_N$. The boundary conditions are therefore
separated, with one natural condition at each endpoint.
\end{proof}

\begin{proof}[Proof of Theorem~\ref{thm:slp_discrete}]
By Lemma~\ref{lem:endpoints}, $\mathsf L_\varphi$ is a self-adjoint
Sturm--Liouville realization with $p=w=\varphi>0$, $q=0$, separated boundary
conditions, and endpoints that are regular or LCNO. Hence
\cite[Thm.~10.12.1(3),(4)]{Zettl} (see also \cite[Thm.~10.6.2]{Zettl})
gives a discrete spectrum bounded below, with simple eigenvalues
\[
\mu_0<\mu_1<\mu_2<\cdots,
\]
and an eigenfunction corresponding to $\mu_j$ has exactly $j$ zeros in
$(0,D)$. Since $\mathcal E$ is nonnegative and vanishes on the constants,
while $\varphi\,dt$ is finite, $\mu_0=0$ and $\mu_j>0$ for $j\ge1$. This proves
Theorem~\ref{thm:slp_discrete}.
\end{proof}

\subsection{The Cheeger energy of a weighted interval}\label{subsec:cheeger}

\begin{lemma}\label{lem:cheeger-identification}
Let $X=([0,D],|\cdot|,c\,\varphi\,ds)$, where $c>0$ and $\varphi$ satisfies
\eqref{eq:concavity} and \eqref{eq:normalization}, and let
\[
0=\lambda_0(X)\le\lambda_1(X)\le\cdots
\]
be the eigenvalues of its Neumann Laplacian. Then the Cheeger energy of $X$
is the quadratic form
\[
\mathcal E_X(u)=c\int_0^D |u'|^2\varphi\,ds,
\]
and
\[
\lambda_j(X)\ge\mu_j(\varphi)
\qquad\text{for every }j\ge0.
\]
\end{lemma}

\begin{proof}
The factor $c$ multiplies both the Cheeger energy and the reference measure,
so it does not affect the Rayleigh quotient. By Lemma~\ref{lem:endpoints},
the operator associated with $\mathcal E$ is $\mathsf L_\varphi$ with the
natural condition $\lim\varphi u'=0$ at both endpoints. The form domain of
the Cheeger energy is contained in $H^1(\varphi)$. Hence, by the min--max
principle,
\[
\lambda_j(X)\ge\mu_j(\varphi)
\qquad\text{for every }j\ge0.
\]
\end{proof}

\bibliographystyle{amsplain}
\bibliography{ev}

\end{document}